\documentclass[12pt]{amsart}
\usepackage{amsmath}
\usepackage{eucal}
\usepackage{amssymb}
\usepackage[all]{xy}
\usepackage{longtable}
\usepackage[osf,sc]{mathpazo}
\usepackage[titletoc]{appendix}
\usepackage{color}
\usepackage{multirow,bigdelim}
\usepackage{stmaryrd}
\usepackage{verbatim}
\usepackage{mathrsfs}
\usepackage{url}
\PassOptionsToPackage{hyphens}{url}\usepackage{hyperref}

\allowdisplaybreaks

\newtheorem{theorem}[equation]{Theorem}
\newtheorem{lemma}[equation]{Lemma}
\newtheorem{proposition}[equation]{Proposition}

\newtheorem{definition}[equation]{Definition}

\newtheorem{example}[equation]{Example}

\newtheorem{remark}[equation]{Remark}

\numberwithin{equation}{subsection}

\allowdisplaybreaks[1]

\newcommand{\LL}{\mathbb{L}}
\newcommand{\TT}{\mathbb{T}}

\newcommand{\CC}{\mathbb{C}}

\newcommand{\bm}{\mathbf{m}}

\newcommand{\bu}{\mathbf{u}}
\newcommand{\bv}{\mathbf{v}}

\DeclareMathAlphabet{\matheur}{U}{eur}{m}{n}

\DeclareMathOperator{\Ker}{Ker} \DeclareMathOperator{\GL}{GL}
\DeclareMathOperator{\Mat}{Mat} 
\DeclareMathOperator{\End}{End} 
\DeclareMathOperator{\Spec}{Spec} 
 \DeclareMathOperator{\Res}{Res}

\DeclareMathOperator{\trdeg}{tr.deg} \DeclareMathOperator{\wt}{wt}

\DeclareMathOperator{\rank}{rank}

\DeclareMathOperator{\Span}{Span}

\newcommand{\oK}{\overline{K}}

\newcommand{\tomega}{\widetilde{\boldsymbol{\omega}}}
\newcommand{\bomega}{\boldsymbol{\omega}}

\newcommand{\tr}{\mathrm{tr}}
\newcommand{\ww}{\boldsymbol{w}}

\begin{document}

\title[]{{\large{O\MakeLowercase{n special values of meromorphic }D\MakeLowercase{rinfeld modular forms of arbitrary rank at }CM\MakeLowercase{ points}}}}

\author{Yen-Tsung Chen}
\address{Department of Mathematics, Pennsylvania State University, University Park, PA 16802, U.S.A.}

\email{ytchen.math@gmail.com}

\author{O\u{g}uz Gezm\.{i}\c{s}}
\address{Department of Mathematics, National Tsing Hua University, Hsinchu City 30042, Taiwan R.O.C.}
\email{gezmis@math.nthu.edu.tw}

\thanks{  }
\thanks{  }


\date{\today}

\begin{abstract} In the present paper, we introduce meromorphic Drinfeld modular forms of arbitrary rank equipped with a particular arithmeticity property. We also study their special values at CM points and show the algebraic independence of these values under some conditions. Our results may be seen as a generalization of Chang's results on the special values of arithmetic Drinfeld modular forms in the rank two setting.
    
\end{abstract}


\maketitle

\section{Introduction}
\subsection{Elliptic modular forms at CM points}
    Let $f$ be an elliptic modular form of weight $\ell\in\mathbb{Z}_{\geq 0}$ for a congruence subgroup of $\mathrm{SL}_2(\mathbb{Z})$. Following the terminology of Shimura, we call $f$ \emph{an arithmetic modular form} if the Fourier expansion of $f$ has algebraic coefficients. An element $\alpha$ in the complex upper half plane $\mathcal{H}$ is called \emph{a CM point} if $\mathbb{Q}(\alpha)$ is an imaginary quadratic extension over $\mathbb{Q}$. Using the result of Shimura \cite[\S6.8]{Shi71}, if $f(\alpha)\neq 0$, then there is a CM elliptic curve $E_\alpha$ defined over the algebraic closure $\overline{\mathbb{Q}}$ of $\mathbb{Q}$ with $\End(E_\alpha)\otimes_{\mathbb{Z}}\mathbb{Q}\cong\mathbb{Q}(\alpha)$ and a non-zero period $\lambda_\alpha$ of $E_\alpha$ such that
    \begin{equation}\label{Eq:Shimura_Period}
        \frac{f(\alpha)}{\lambda_\alpha^\ell}\in\overline{\mathbb{Q}}^\times.
    \end{equation}
    In fact, the ratio $f(\alpha)/\lambda_\alpha^\ell$ lies in the finite extension of $\mathbb{Q}$, which is closely related to the class field theory for $\mathbb{Q}(\alpha)$. We refer the reader to \cite[\S6]{Shi71} for further details.

    As Shimura's result builds up a bridge between special CM values of arithmetic modular forms and periods of CM elliptic curves defined over $\overline{\mathbb{Q}}$, one can apply tools from the transcendence theory to investigate algebraic relations among values of arithmetic modular forms at CM points. More precisely, by the transcendence result of Schneider \cite{Sch37} together with \eqref{Eq:Shimura_Period}, it can be shown that $f(\alpha)$ is either zero or transcendental (cf.  \cite{GMR11}) provided that $j(\alpha)\in\overline{\mathbb{Q}}^\times$ where $j(\cdot):\mathcal{H}\to\mathbb{C}$ is the classical modular $j$-function. A generalization of this result to meromorphic modular forms has been established recently in \cite{BhPa25}. In addition, based on the result of Wolfart and W\"ustholz in \cite{WW85} and \eqref{Eq:Shimura_Period}, if the given arithmetic modular form $f$ is of weight $1$, then whenever $\alpha_1,\dots,\alpha_n\in\mathcal{H}$ are distinct CM points with $f(\alpha_i)\neq 0$, we have
    \[
        \dim_{\overline{\mathbb{Q}}}\mathrm{Span}_{\overline{\mathbb{Q}}}\{f(\alpha_1),\dots,f(\alpha_n)\}=n.
    \]

    Motivated by these classical results, the main theme of this article is two-fold. We first investigate an analogue of Shimura's result \eqref{Eq:Shimura_Period} to arithmetic Drinfeld modular forms of arbitrary rank (Theorem \ref{Intro:Special_Values_to_Periods}). Then we will determine algebraic relations among periods and quasi-periods of CM Drinfeld modules whose endomorphism algebras are Galois and linearly disjoint over the base field which, via Theorem \ref{Intro:Special_Values_to_Periods}, will lead to an algebraic independence result for special values of arithmetic Drinfeld modular forms at CM points (Theorem \ref{Thm:Intro}).
    
\subsection{Drinfeld modular forms at CM points}
    Let $A=\mathbb{F}_q[\theta]$ be the polynomial ring in variable $\theta$ over the finite field $\mathbb{F}_q$ of characteristic $p>0$. We set $K:=\mathbb{F}_q(\theta)$ to be the fraction field of $A$ and define the non-Archimedean absolute value $|\cdot|_\infty$ on $K$ by setting $|f/g|_\infty:=q^{\deg_\theta f-\deg_\theta g}$. Then we identify the completion of $K$ with respect to $|\cdot|_\infty$ with $K_\infty=\mathbb{F}_q(\!(\theta^{-1})\!)$ and set $\mathbb{C}_\infty$ to be the completion of an algebraic closure of $K_\infty$. We further let $\oK$ be the fixed algebraic closure of $K$ in $\mathbb{C}_{\infty}$.

    The theory of Drinfeld modular forms dates back to the Ph.D. thesis of D. Goss (see also \cite{Gos80}), where the analytical and algebraic theory of  Drinfeld modular forms of rank $2$ were introduced. Later, Gekeler made significant contributions to the theory in his series of works \cite{Gek84,Gek86,Gek88,Gek89} in the rank $2$ case. Beyond the rank $2$ setting, Basson, Breuer and Pink initiated a project \cite{BBP21} aiming to build up the foundation of the theory of Drinfeld modular forms of arbitrary rank $r\geq 2$ based on the algebraic Satake compactification of Drinfeld moduli spaces introduced by Pink \cite{Pin13} as well as the analytic Satake compactification established by H\"{a}berli \cite{Hab21}. At the same time, Gekeler established a series of works \cite{Gek17,Gek18,Gek19, Gek25} that also enlighten the development of the theory of Drinfeld modular forms of higher rank.

    For $r\geq 2$, we consider the Drinfeld upper half plane
    \[
        \Omega^r:=\mathbb{P}^{r-1}(\mathbb{C}_\infty)\setminus\{K_\infty\text{-rational }\mathrm{hyperplanes}\}.
    \]
    We may identify any element $\bomega\in\Omega^r$ with the column vector $(w_1,\dots,w_r)^\tr\in\Mat_{r\times 1}(\mathbb{C}_\infty)$ whose entries are $K_{\infty}$-linearly independent and normalized so that $w_r=1$. Note that each $\bomega=(w_1,\dots,w_r)^\tr\in\Omega^r$ induces \emph{an $A$-lattice}, that is a free $A$-module $\Lambda_{\bomega}=Aw_1+\cdots+Aw_r\subset\mathbb{C}_\infty$ of rank $r$ such that the intersection of $\Lambda_{\bomega}$ with any ball of finite radius in $\mathbb{C}_{\infty}$ is finite (see \S2.1 for more details). The endomorphism ring $\End(\Lambda_{\bomega}) $ of $\Lambda_{\bomega}$ is given by
    \[
        \End(\Lambda_{\bomega}):=\{c\in\mathbb{C}_\infty\mid c\Lambda_{\bomega}\subset\Lambda_{\bomega}\}\subset\mathbb{C}_\infty,
    \]
    which is a free $A$-module of rank $s\geq 1$ with $s\mid r$. We call $\bomega$ \emph{a CM point} if $\rank_A\End(\Lambda_{\bomega})=r$. It is known that $\Omega^r$ has a rigid analytic space structure (see \cite{BBP21,FvdP04} for more details). 
    
    For a congruence subgroup $\Gamma\subset\GL_r(A)$, \emph{a weak Drinfeld modular form of weight $k\in\mathbb{Z}$ and type $m\in\mathbb{Z}/(q-1)\mathbb{Z}$ for $\Gamma$} is a $\mathbb{C}_\infty$-valued rigid analytic function $F:\Omega^r\to\mathbb{C}_\infty$ satisfying the automorphy condition for elements in $\Gamma$. It is called \emph{a Drinfeld modular form of weight $k\in\mathbb{Z}$ and type $m\in\mathbb{Z}/(q-1)\mathbb{Z}$ for $\Gamma$} if it further satisfies a certain growth condition around the cups of $\Gamma$ (see \S2.2 for their precise definition). In fact, there is a parameter at infinity $u_\Gamma:\Omega^r\to\mathbb{C}_\infty$ such that any Drinfeld modular form $F$ has a unique \emph{$u_\Gamma$-expansion}, an analogue of the Fourier expansion. More precisely, we write
    \[
        F(\boldsymbol{\omega})=\sum_{n\geq 0}f_n(\widetilde{\boldsymbol{\omega}})u_\Gamma(\boldsymbol{\omega})^n,~\bomega=(w_1,\widetilde{\bomega}^\tr)^\tr\in\Omega^r
    \]
    where $f_n:\Omega^{r-1}\to\mathbb{C}_\infty$ is a uniquely determined rigid analytic function on $\Omega^{r-1}$ which can be regarded as an analogue of the Fourier coefficients for $F$ and $\tomega\in \Omega^r$ lies in some neighborhood of infinity. When $r=2$, we note that each $f_n$ is a constant in $\mathbb{C}_{\infty}$. 

    In the rank $2$ setting, one may naturally define arithmetic Drinfeld modular forms to be any Drinfeld modular form $F$ satisfying $f_n\in\oK$. Using the result of Gekeler \cite{Gek84,Gek86}, Chang \cite[Thm.~2.2.1]{Cha12} established an analogue of \eqref{Eq:Shimura_Period} for values of arithmetic Drinfeld modular forms at CM points. The first obstruction to investigating \eqref{Eq:Shimura_Period} in the case of $r\geq 3$ is that the $u_\Gamma$-coefficients $f_n$ are no longer elements in $\mathbb{C}_\infty$ but a $\mathbb{C}_\infty$-valued rigid analytic functions on $\Omega^{r-1}$. Thus, one may not have an immediate definition for an arithmetic Drinfeld modular form in the higher rank setting. To overcome this problem, we instead adopt the recursive definition initiated by Basson in his PhD thesis \cite{Bas14}. First, we note by \cite[Prop. 3.2.7]{Bas14} that,  the $u_\Gamma$-coefficients $f_n$ of a given Drinfeld modular form $F$ for $\Gamma\subset \GL_r(A)$ are in fact weak Drinfeld modular forms for some congruence subgroup $\widetilde{\Gamma}\subset\GL_{r-1}(A)$. Then we call $F$ \emph{arithmetic} if each $f_n:\Omega^{r-1}\to \mathbb{C}_{\infty}$ is an arithmetic weak Drinfeld modular form. We remark that, previously, this definition was used by Sugiyama \cite{Sug18} in his study of the integrality and congruence relations among Drinfeld modular forms for $\GL_r(A)$ (see Example \ref{Ex:2}). Furthermore, we call $F$ \emph{a meromorphic arithmetic Drinfeld modular form of weight $k$ for $\Gamma$} if it is a ratio of two arithmetic Drinfeld modular forms $f$ and $g$ for $\Gamma$ of weight $k+\ell\in\mathbb{Z}_{\geq 1}$ and $\ell\in\mathbb{Z}_{\geq 1}$ respectively.
    
    The first purpose in this article is to study meromorphic arithmetic Drinfeld modular forms of arbitrary rank for any congruence subgroup and their special values at CM points. Before stating our theorem, fixing a $(q-1)$-st root of $-\theta$, we define \emph{the Carlitz period} $\widetilde{\pi}$ by the infinite product
    \[
        \widetilde{\pi}:=-(-\theta)^{q/(q-1)}\prod_{i=1}^\infty\left(1-\theta^{1-q^i}\right)^{-1}\in\mathbb{C}_\infty^\times.
    \]

    Let $N\in A\setminus\mathbb{F}_q$ and $\mathcal{T}_N:=(N^{-1}A/A)^r\setminus\{0\}$. Let $F$ be an arithmetic Drinfeld modular form of weight $k$ for $\Gamma\supset\Gamma(N)$. We prove in Proposition~\ref{P:coefficients} that $F$ is integral over $\oK[E_{\bu}^{\mathrm{ari}}\mid\bu\in\mathcal{T}_\theta]$,
    where each $E_\bu^{\mathrm{ari}}$ is a particular arithmetic Drinfeld modular form of weight $1$ for $\Gamma(\theta)$ (see \S3 for more details). Lastly, given $\alpha,\beta\in\mathbb{C}_\infty$, we further denote by $\alpha\sim\beta$ if $\alpha/\beta\in\oK$. Our first result, which will be restated in Theorem~\ref{Thm:Special_Values_to_Periods} later, is given as follows.
    \begin{theorem}\label{Intro:Special_Values_to_Periods}
        Let $F$ be a meromorphic arithmetic Drinfeld modular form of non-zero weight $k$ for a congruence subgroup $\Gamma$ of $\GL_r(A)$. For any CM point $\boldsymbol{\omega}\in\Omega^r$, if $F$ is defined at $\boldsymbol{\omega}$, then we have
        \[
            F(\boldsymbol{\omega})\sim\left(\lambda_{\boldsymbol{\omega}}/\widetilde{\pi}\right)^k
        \]
        where $\lambda_{\boldsymbol{\omega}}\in\mathbb{C}_\infty^\times$ is a non-zero period of a CM Drinfeld module $\psi^{\boldsymbol{\omega}}$ defined over $\oK$ whose period lattice is homothetic to $\Lambda_{\bomega}$.
    \end{theorem}

    \begin{remark}
        \textnormal{We note that Ayotte proved in his thesis \cite{Ayo23} another analogue of \eqref{Eq:Shimura_Period} for arithmetic Drinfeld modular forms at a certain family of CM points by following the strategy of Urban \cite{Urb14}. However, those forms are defined through the algebraic description of Drinfeld modular forms (see \cite[Def.~3.4.13]{Ayo23}) and hence their definition is different than ours. Also, Ayotte only considered CM points $\bomega$ such that the infinite place is inert in the fraction field of $\End(\Lambda_{\bomega})$. It would be interesting to compare our result Theorem~\ref{Intro:Special_Values_to_Periods} and \cite[Thm.~4.6.1]{Ayo23}.}
    \end{remark}

\subsection{Periods and quasi-periods of non-isogenous CM Drinfeld modules}
    Let $\phi$ be a Drinfeld module of rank $r$ defined over $\oK$ with the period lattice $\Lambda_\phi\subset\mathbb{C}_\infty$. It has the period matrix $P_\phi$, given as in \eqref{Eq:Period_Matrix}, that consists of periods and quasi-periods of $\phi$. If we denote by $\oK(P_\phi)$ the field generated by the entries of $P_\phi$ over $\oK$, then its transcendence degree was determined completely in \cite[Thm.~1.2.2]{CP12}. More precisely, if we define $s:=\End(\Lambda_\phi)$, then we have 
    \[
    \trdeg_{\oK}\oK(P_\phi)=r^2/s.
    \]
    
    In what follows, we describe Chang's result \cite[Thm.~2.2.2]{Cha12} on the transcendence theory for periods and quasi-periods of several CM Drinfeld modules. Consider $\phi_1,\dots,\phi_n$ to be CM Drinfeld modules of rank $2$ defined over $\oK$ which are pairwise non-isogenous (see \S2.1 for the definition of isogeny between Drinfeld modules). By the Legendre relation for Drinfeld modules (see \eqref{Eq:Legendre_Relation}), we have $\det(P_{\phi_i})/\det(P_{\phi_j})\in\oK^\times$. The result of Chang asserts that all the  algebraic relations among periods and quasi-periods of $\phi_1,\dots,\phi_n$ are those coming from the Legendre relation. In particular, he proved that
    \[
        \trdeg_{\oK}\oK(\cup_{i=1}^nP_{\phi_i})=n+1.
    \]

    The second main result in this paper, which will be restated in Theorem~\ref{Thm:Algebraic_Independence} later, is a generalization of Chang's result to CM Drinfeld modules defined over $\oK$ of arbitrary rank whose endomorphism algebras are Galois over $\mathbb{F}_q(t)$.
    \begin{theorem}\label{Intro:Algebraic_Independence}
        Let $\phi_1,\dots,\phi_n$ be CM Drinfeld modules of rank $r_1,\dots,r_n\geq 2$ defined over $\oK$. Suppose that for any $1\leq i \leq n$, the endomorphism algebra $\mathcal{K}_i:=\mathbb{F}_q(t)\otimes_{\mathbb{F}_q[t]}\End(\Lambda_{\phi_i})$ is Galois over $\mathbb{F}_q(t)$ and $\mathcal{K}_i\cap\mathcal{K}_j=\mathbb{F}_q(t)$ for any $i\neq j$. Let $\oK\big(\cup_{i=1}^nP_{\phi_i}\big)$ be the field generated by the entries of $P_{\phi_i}$ for each $1\leq i\leq n$ over $\oK$. Then we have
        \[
            \trdeg_{\oK}\oK(\cup_{i=1}^nP_{\phi_i})=(r_1+\cdots+r_n)-(n-1).
        \]
    \end{theorem}
    
    \begin{remark} \textnormal{We note that our conditions on the endomorphism algebras $\mathcal{K}_1,\dots,\mathcal{K}_n$ imply that the collection $\{\phi_1,\dots,\phi_n\}$ of CM Drinfeld modules is pairwise non-isogenous (\cite[Prop. 2.1.1]{Cha12}).}    
    \end{remark}

    \begin{remark} \textnormal{At the writing of this article, the authors were informed that Maurischat and Namoijam investigated the transcendence degree of hyperderivatives of periods and quasi-periods of direct sums of Drinfeld modules by studying the image of the Galois representations. It would be interesting to compare this approach with the methods of the present paper.}
    \end{remark}

    The key ingredient of our approach is to apply the theory developed by Papanikolas in \cite{Pap08} and to determine the dimension of the $t$-motivic Galois group attached to the direct sum of non-isogenous CM Drinfeld modules $\phi_1,\dots,\phi_n$. The idea of our proof is inspired by the argument used in \cite[Thm.~3.3.2]{Cha12}. However, our approach still differs from his in a certain aspect. In \cite{Cha12}, his induction hypothesis to obtain a contradiction at the end relies on the following fact, which mainly follows from the assumption  $r_i=2$ for each $1\leq i\leq n$ : If $\trdeg_{\oK}\oK(\cup_{i=1}^{n-1}P_{\phi_i})=n$ but $\trdeg_{\oK}\oK(\cup_{i=1}^nP_{\phi_i})<n+1$, then we must have $\trdeg_{\oK}\oK(\cup_{i=1}^nP_{\phi_i})=n$. On the other hand, in our setting, due to the arbitrary choice of $r_i$, if  we assume that $\trdeg_{\oK}\oK(\cup_{i=1}^{n-1}P_{\phi_i})=(r_1+\cdots+r_{n-1})-(n-2)$ but $\trdeg_{\oK}\oK(\cup_{i=1}^nP_{\phi_i})<(r_1+\cdots+r_{n})-(n-1)$, then we cannot directly determine $\trdeg_{\oK}\oK(\cup_{i=1}^nP_{\phi_i})$ as there are several possibilities for the transcendence degree. Therefore, Chang's method cannot be transported to our situation without any innovation. To overcome this problem, we apply some properties of algebraic tori. More precisely, since the $t$-motivic Galois group $G$ associated to the direct sum of $\phi_1,\dots,\phi_n$ is an algebraic torus (see \eqref{Eq:Galois_1}), as described in the proof of Theorem \ref{Thm:Galois}, we can find a subtori $H$ and $H'$ of $G$ such that $G=H\cdot H'$ and $|H\cap H'|<\infty$. Here, we note that the dimension of $H'$ is determined by $r_1,\dots,r_{n-1}$. Then using the condition that the endomorphism algebras of $\phi_1,\dots,\phi_n$ are pairwise linearly disjoint, it follows that the maximal anisotropic subtorus of $H'$ actually lies in the kernel of the projection of $G$ into the $t$-motivic Galois group corresponding to $\phi_n$. This allows us to determine the dimension of $G$. We refer the reader to \S4 for all these details.
    
    As a consequence of Theorem~\ref{Intro:Special_Values_to_Periods} and Theorem~\ref{Intro:Algebraic_Independence} as well as a detailed analysis on the transcendence basis of $\oK(\cup_{i=1}^nP_{\phi_i})$, we deduce our next result which extends \cite[Thm.~1.2.1]{Cha12} to meromorphic arithmetic Drinfeld modular forms of arbitrary rank.
    \begin{theorem}\label{Thm:Intro}
        Let $\bomega_1,\dots,\bomega_n\in\Omega^r$ be CM points such that
        \begin{enumerate}
            \item $\mathcal{K}_i:=\mathbb{F}_q(t)\otimes_{\mathbb{F}_q[t]}\End(\Lambda_{\bomega_i})$ is Galois over $\mathbb{F}_q(t)$ for each $1\leq i\leq n$.
            \item $\mathcal{K}_i\cap\mathcal{K}_j=\mathbb{F}_q(t)$ for any $i\neq j$.
        \end{enumerate}
        Let $F$ be a meromorphic arithmetic Drinfeld modular form for a congruence subgroup of $\GL_r(A)$. Assume that  $F$ is defined at $\bomega_i$ and $F(\bomega_i)\neq 0$ for every $1\leq i\leq n$. Then we have
        \[
            \trdeg_{\oK}\oK(F(\bomega_1),\dots,F(\bomega_n))=n.
        \]
    \end{theorem}
    
    The outline of the paper can be described as follows. In \S2, we introduce preliminaries on Drinfeld modules, quasi-periodic functions, dual $t$-motives as well as Drinfeld modular forms. In \S3, we discuss the notion of (meromorphic) arithmetic Drinfeld modular forms and describe their integrality in Proposition \ref{P:coefficients}. Later on, we provide a proof for Theorem \ref{Thm:Special_Values_to_Periods}. In \S4, we analyze some properties of algebraic tori and briefly review the transcendence theory of Papanikolas established in \cite{Pap08}. Finally, we provide a proof for Theorem \ref{Thm:Intro}. 

    \subsection*{Acknowledgments} The authors are grateful to Gebhard B\"{o}ckle and Sriram Chinthalagiri Venkata for their valuable comments and feedback which improve Theorem~\ref{Thm:Special_Values_to_Periods}. The first author was partially supported by the AMS-Simons Travel Grants. The second author acknowledges support from NSTC Grant 113-2115-M-007-001-MY3.    
    
\section{Preliminaries and Background}
In this section, we collect several fundamental properties of Drinfeld modules, quasi-periods and quasi-periodic functions of Drinfeld modules as well as Drinfeld modular forms which will be necessary to prove our results in later sections. Our exposition is based on \cite{Pap23, BBP21}.
\subsection{Drinfeld modules and quasi-periodic functions}
    Let $K\subseteq L \subseteq \mathbb{C}_{\infty}$ be field. We define $L[\tau]$ to be the twisted polynomial ring over $L$ subject to the relation that $\alpha^q\tau=\tau\alpha$ for $\alpha\in L$. One can define an action of $L[\tau]$ on $L$ given by
    \[
        \left(\sum_{i=0}^nc_i\tau^i\right)\cdot\alpha:=\sum_{i=0}^nc_i\alpha^{q^i}\in L,~\alpha\in L.
    \]
    We consider $t$ to be an indeterminate over $\mathbb{C}_{\infty}$ and for any $a\in \mathbb{F}_q[t]$, we set $a(\theta):=a|_{t=\theta}\in A$. Let $r$ be a positive integer. \emph{A Drinfeld module of rank $r$ defined over $L$} is an $\mathbb{F}_q$-algebra homomorphism 
    \begin{align*}
        \phi:\mathbb{F}_q[t]&\to L[\tau]\\
        a&\mapsto\phi_a=a(\theta)+g_{1,a}\tau+\dots+g_{r\deg(a),a}\tau^{r\deg(a)}
    \end{align*}
    such that $g_{r\deg(a),a}\neq 0$. We say that Drinfeld modules $\phi$ and $\phi'$ are \emph{isogenous} if there exists $u\in \mathbb{C}_{\infty}[\tau]$ such that $\phi_t u=u\phi'_t$. Moreover, we say that $\phi$ and $\phi'$ are \emph{isomorphic}, denoted by $\phi\sim \phi'$, if $u\in \mathbb{C}_{\infty}^{\times}$.  We denote by
 $ 
        \End(\phi)$
    the ring of isogenies of $\phi$. It has a natural $\mathbb{F}_q[t]$-module structure given by $a\cdot u:=\phi_au$ for $a\in\mathbb{F}_q[t]$ and $u\in\End(\phi)$.
    
    For each Drinfeld module $\phi$, there is a unique $\mathbb{F}_q$-linear power series $\exp_\phi(X)=\sum_{i\geq 0}\alpha_iX^{q^i}\in\mathbb{C}_\infty\llbracket X\rrbracket$ so that $\alpha_0=1$ and $\exp_{\phi}$ induces an entire function $\exp_\phi:\mathbb{C}_\infty\to\mathbb{C}_\infty$ satisfying the functional equation
    \begin{equation}\label{E:exp}
        \phi_t\big(\exp_\phi(X)\big)=\exp_\phi\big(\theta X\big).
    \end{equation}
   We set $\Lambda_\phi:=\Ker(\exp_\phi)$ and call each non-zero element of $\Lambda_{\phi}$ \emph{a period of $\phi$}. We also call $\Lambda_{\phi}$  \emph{the period lattice of $\phi$}. 

    A free $A$-module of rank $r$ is called \emph{an $A$-lattice} $\Lambda$ of rank $r$ if its intersection with any ball in $\CC_{\infty}$ of finite radius is finite. We define the \emph{exponential function $e_{\Lambda}:\CC_{\infty}\to \CC_{\infty}$ of $\Lambda$} by the infinite product
    \[
        e_{\Lambda}(\omega):=\omega\prod_{\substack{\lambda\in \Lambda\\\lambda\neq 0}}\Big(1-\frac{\omega}{\lambda}\Big), \ \ \omega\in \CC_{\infty}.
    \]
Let $\Lambda'$ be an $A$-lattice. We say that $\Lambda$ and $\Lambda'$ are \emph{isogenous} if there exists $\alpha\in \mathbb{C}_{\infty}^{\times}$ such that $\alpha\Lambda\subseteq \Lambda'$ and $\Lambda'/\alpha\Lambda$ is a finite $A$-module. Moreover, we say that $\Lambda$ and $\Lambda'$ are \emph{homothetic} if $\alpha\Lambda=\Lambda'$. 

By Drinfeld \cite{Dri74}, there exists an equivalence of categories between the category of 
Drinfeld modules of rank $r$ defined over $\mathbb{C}_{\infty}$ and the category of $A$-lattices of rank $r$. More precisely, for any Drinfeld module $\phi$ of rank $r$, there exists a unique $A$-lattice $\Lambda$ of rank $r$, which is the period lattice of $\phi$. Conversely, each $A$-lattice $\Lambda$ of rank $r$ corresponds to a Drinfeld module $\phi^{\Lambda}$ of rank $r$ defined over $\mathbb{C}_{\infty}$ so that $\exp_{\phi^{\Lambda}}=e_{\Lambda}$. One of the most fundamental examples of Drinfeld modules is \emph{the Carlitz module $C$} defined by 
\[
C_t:=\theta+\tau.
\]
It corresponds to the $A$-lattice of rank one generated by the Carlitz period $\widetilde{\pi}$ over $A$.
    
   We consider the endomorphism ring $\End(\Lambda_\phi)$ of $\Lambda_\phi$ given by
    \[
        \End(\Lambda_\phi):=\{c\in\mathbb{C}_\infty\mid c\Lambda_\phi\subset\Lambda_\phi\}.
    \]
    It is equipped with a natural $\mathbb{F}_q[t]$-module structure given by the left multiplication. In fact, it is a free $\mathbb{F}_q[t]$-module of rank $s$ with $1\leq s\leq r$ with $s|r$ and $\End(\Lambda_\phi)\cong\End(\phi)$ (see, for example, \cite[Thm.~5.2.11]{Pap23}). 

    
    An $\mathbb{F}_q$-linear map $\delta:\mathbb{F}_q[t]\to\mathbb{C}_\infty[\tau]\tau$ is called \emph{a $\phi$-biderivation} if $\delta(ab)=a(\theta)\delta(b)+\delta(a)\phi_b$ for all $a,b\in\mathbb{F}_q[t]$. Given a $\phi$-biderivation $\delta$, there exists a unique $\mathbb{F}_q$-linear power series ${\rm{F}}^{\phi}_\delta(X)\in X^q\mathbb{C}_\infty\llbracket X\rrbracket$ such that for any $a\in\mathbb{F}_q[t]$
    \begin{equation}\label{Eq:Period_Matrix}
        {\rm{F}}^{\phi}_\delta(a(\theta)X)-a(\theta){\rm{F}}^{\phi}_\delta(X)=\delta_a\big(\exp_\phi(X)\big).
    \end{equation}
    The power series ${\rm{F}}^{\phi}_\delta(X)$ induces an entire function ${\rm{F}}^{\phi}_\delta:\mathbb{C}_{\infty}\to \mathbb{C}_{\infty}$ which we call \emph{the quasi-periodic function associated to $\delta$}. The values ${\rm{F}}^{\phi}_\delta(\omega)$ for $\omega\in\Lambda_\phi$ are called \emph{quasi-periods of $\phi$}. Note that each $\phi$-biderivation $\delta$ is uniquely determined by $\delta(t)$. We consider a $\phi$-biderivation defined by the rule $t\to \phi_t-\theta$ and denote by $F_{\tau^0}^{\phi}$ the quasi-periodic function associated to it. Then ${\rm{F}}^{\phi}_{\tau^0}(X)=\exp_\phi(X)-X$ and ${\rm{F}}^{\phi}_{\tau^0}(\omega)=-\omega$ for $\omega\in\Lambda_\phi$. For $1\leq i\leq r-1$, we further set $F_{\tau^i}^{\phi}$ to be the quasi-periodic function associated to the $\phi$-biderivation defined by the rule $t\to \tau^i$. Let $w_1,\dots,w_r$ be a fixed $A$-basis for $\Lambda_\phi$. Then we define
    \[
        P_\phi:=\begin{pmatrix}
            -w_1 & {\rm{F}}^{\phi}_{\tau}(w_1) & \cdots & {\rm{F}}^{\phi}_{\tau^{r-1}}(w_1)\\
            -w_2 & {\rm{F}}^{\phi}_{\tau}(w_2) & \cdots & {\rm{F}}^{\phi}_{\tau^{r-1}}(w_2)\\
            \vdots & \vdots & & \vdots\\
            -w_r & {\rm{F}}^{\phi}_{\tau}(w_r) & \cdots & {\rm{F}}^{\phi}_{\tau^{r-1}}(w_r)
        \end{pmatrix}\in \Mat_r(\mathbb{C}_{\infty})
    \]
    and call $P_{\phi}$ \emph{the period matrix of $\phi$}. Let $\oK(P_\phi)$ be the field generated by entries of $P_\phi$ over $\oK$. When $\phi$ is a Drinfeld module of rank $r$ defined over $\oK$, the transcendence degree of $\oK(P_\phi)$ can be completely determined (\cite[Thm.~1.2.2]{CP12}). More precisely, letting $s:=\rank_{\mathbb{F}_q[t]}\End(\phi)$, we have
    \[
        \trdeg_{\oK}\oK(P_\phi)=\frac{r^2}{s}.
    \]
    


\subsection{Drinfeld modules with complex multiplication}
    Let $\phi$ be a Drinfeld module of rank $r$. We  call $\phi$ \emph{a CM Drinfeld module of rank $r$} if $\End(\Lambda_\phi)$ is free of rank $r$ over $\mathbb{F}_q[t]$. In this subsection, we discuss some properties of CM Drinfeld modules defined over $\oK$ which will be useful in \S4.2. For more details, we refer the reader to \cite[\S7.5]{Pap23}. 

    Let $\phi$ be a CM Drinfeld module defined over $\oK$ with the period lattice $\Lambda_{\phi}$. We choose a period $\lambda\in \Lambda_{\phi}$. By \cite[Lem. 3.1.1]{CG25} (cf. \cite[(3.12), (3.13)]{BrPa02}), 
    if we fix an $A$-basis $\{\eta_1,\dots,\eta_r\}$ for $\Lambda_{\phi}$, then for each $1\leq i\leq r$ and $0\leq j\leq r-1$, there is a homegenous polynomial $\mathcal{L}_{ij}(X_0,\dots,X_{r-1})\in\oK[X_0,\dots,X_{r-1}]$ of degree one such that
    \[
        F_{\tau^j}^{\phi}(\eta_i)=\mathcal{L}_{ij}\big(\lambda,F_{\tau}^{\phi}(\lambda),\dots,F_{\tau^{r-1}}^{\phi}(\lambda)\big).
    \]
    We set
    \[
        \mathbf{D}(X_0,\dots,X_{r-1}):=\det\begin{pmatrix}
            \mathcal{L}_{10}&\dots &\mathcal{L}_{1(r-1)}\\
            \vdots & & \vdots \\
            \mathcal{L}_{r0}&\dots &\mathcal{L}_{r(r-1)}
        \end{pmatrix}\in\overline{K}[X_0,\dots,X_{r-1}].
    \]
    By \cite[(3.8)]{CG22a}, there exists $c_\phi\in\oK^\times$ such that
    \begin{equation}\label{Eq:Legendre_Relation}
        \mathbf{D}(\lambda,F_{\tau}^{\phi}(\lambda),\dots,F_{\tau^{r-1}}^{\phi}(\lambda))=c_\phi\widetilde{\pi}.
    \end{equation}
    We call the identity \eqref{Eq:Legendre_Relation} \emph{the Legendre relation for $\phi$}. For each $0\leq j\leq r-1$, consider the polynomial
    \begin{align*}
        \mathbf{Q}_j(X_j)&:=\bold{D}(\lambda,F_{\tau}^{\phi}(\lambda),\dots,X_j,\dots,F_{\tau^{r-1}}^{\phi}(\lambda))-c_\phi\widetilde{\pi}\\
        &\in\oK(\lambda,F_{\tau}^{\phi}\big(\lambda),\dots,\widehat{F_{\tau^j}(\lambda)},\dots,F_{\tau^{r-1}}^{\phi}(\lambda),\widetilde{\pi}\big)[X_j].
    \end{align*}
    Here, by the notation $\widehat{F_{\tau^j}(\lambda)}$, we mean that $F_{\tau^j}(\lambda)$ is not a generator of the field described. By the Legendre relation for $\phi$, we see that at least one of $Q_j(X_j)$ is not a zero polynomial. For $1\leq \mu \leq r-1$, we say that \emph{$F_{\tau^\mu}(\lambda)$ is expressible by the Legendre relation for $\phi$} if $Q_\mu(X_\mu)$ is not a zero polynomial.     
      In this case, $F_{\tau^\mu}(\lambda)$ must be algebraic over the field $\oK(\lambda,F_{\tau}^{\phi}\big(\lambda),\dots,\widehat{F_{\tau^\mu}(\lambda)},\dots,F_{\tau^{r-1}}^{\phi}(\lambda),\widetilde{\pi}\big)$. Furthermore, we have
    \[
\trdeg_{\oK}\oK(\lambda,F_{\tau}^{\phi}\big(\lambda),\dots,F_{\tau^{r-1}}^{\phi}(\lambda)\big)=\trdeg_{\oK}\oK(\lambda,F_{\tau}^{\phi}\big(\lambda),\dots,\widehat{F_{\tau^\mu}^\phi(\lambda)},\dots,F_{\tau^{r-1}}^{\phi}(\lambda),\widetilde{\pi}\big)=r.
    \]
    We call $(\phi,\lambda)$ \emph{a bad pair} if none of $F^{\phi}_{\tau}(\lambda),\dots,F^{\phi}_{\tau^{r-1}}(\lambda) $ is  expressible by the Legendre relation for $\phi$. We note that, in this case, $Q_0(X_0)$ is not a zero polynomial.
    
    The following lemma shows that if $\phi_1$ and $\phi_2$ are non-isogenous Drinfeld modules, then $(\phi_1,\lambda_1)$ and $(\phi_2,\lambda_2)$ can not be bad pairs simultaneously.

    \begin{lemma}\label{L:badpair}
        Let $\phi$ be a CM Drinfeld module of rank $r$ defined over $\oK$. Assume that $\lambda$ is a period of $\phi$ such that $(\phi,\lambda)$ is a bad pair. Then we have $\lambda^r/\widetilde{\pi}\in\oK^\times$. In particular, if $(\phi_1,\lambda_1)$ and $(\phi_2,\lambda_2)$ are two bad pairs, then $\phi_1$ and $\phi_2$ are isogenous.
    \end{lemma}

    \begin{proof}
        Note that since $\mathcal{L}_{ij}$ is a homogeneous polynomial in $X_0,\dots,X_{r-1}$ for each $1\leq i\leq r$ and $0\leq j\leq r-1$, $\mathbf{D}(X_0,\dots,X_{r-1})$ is also a homogeneous polynomial of degree $r$. Since $(\phi,\lambda)$ is a bad pair, for each $1\leq j\leq r-1$, the polynomial
        \[
            \mathbf{Q}_j(X_j)=\mathbf{D}(\lambda,F_{\tau}^{\phi}(\lambda),\dots,X_j,\dots,F_{\tau^{r-1}}^{\phi}(\lambda))-c_\phi\widetilde{\pi}
        \]
        is a zero polynomial. It implies that the Legendre relation \eqref{Eq:Legendre_Relation} is independent of $F_{\tau^\mu}^\phi(\lambda)$ for each $1\leq \mu\leq r-1$. Hence $\mathbf{D}(X_0,\dots,X_{r-1})$ is simply a homogeneous polynomial of degree $r$ in variable $X_0$, which is an algebraic multiple of the monomial $X_0^r$. In this case, \eqref{Eq:Legendre_Relation} implies $\lambda^r/\widetilde{\pi}\in\oK^\times$ which gives the first assertion. We prove the second assertion. By the first assertion, note that, if $(\phi_1,\lambda_1)$ and $(\phi_2,\lambda_2)$ are two bad pairs, then we have $\lambda_1/\lambda_2\in\oK^\times$. By \cite[Thm.~1.2]{Yu90}, this implies that $\phi_1$ and $\phi_2$ are isogenous as desired.
    \end{proof}

\subsection{Rigid analytic trivialization and dual $t$-motives}
    Our goal in this subsection is to review  \emph{the $t$-motivic objects} associated to a Drinfeld module $\phi$ defined over $\oK$. For a through discussion on these objects, we refer the reader to \cite{BP20, CP12, Pap08}.
    
    For $n\in\mathbb{Z}$, we define the $n$-fold Frobenius twisting on the Laurent series field $\mathbb{C}_\infty(\!(t)\!)$ by setting
    \begin{align}
        \begin{split}
            (\cdot)^{(n)}:\mathbb{C}_\infty(\!(t)\!)&\to\mathbb{C}_\infty(\!(t)\!)\\
        f=\sum c_it^i&\mapsto f^{(n)}:=\sum c_i^{q^n}t^i.
        \end{split}
    \end{align}
Furthermore, we extend the operation $ (\cdot)^{(n)} $ to any matrix $B=(b_{ij})_{ij}\in \Mat_{r}(\mathbb{C}_\infty(\!(t)\!))$ by setting $B^{(n)}:= (b_{ij}^{(n)})_{ij}\in \Mat_{r}(\mathbb{C}_\infty(\!(t)\!))$. Moreover,  for any $c\in \mathbb{C}_{\infty}^\times$, we let $\TT_c$ be \emph{the Tate algebra} on the closed disk of radius $|c|_\infty$ centered at the origin, that is
    \[
        \TT_c:=\{f:=\sum_{n\geq 0}a_it^i\in\mathbb{C}_\infty\llbracket t\rrbracket\mid f~\mathrm{converges~at~each}~\alpha~\mathrm{with}~|\alpha|_\infty\leq |c|_{\infty}\}.
    \]
    For simplicity, we denote by $\mathbb{T}$ the Tate algebra $\mathbb{T}_1$, the ring of $\mathbb{C}_{\infty}$-valued functions converging on the unit disk.
    
    Let $\oK[t,\sigma]:=\oK[t][\sigma]$ be the twisted polynomial ring in variable $\sigma$ with coefficients in $\oK[t]$ subject to the relation $\sigma f=f^{(-1)}\sigma$ for $f\in \oK[t]$. 
    Let $\mathcal{M}$ be a left $\oK[t,\sigma]$-module. We call $\mathcal{M}$ \emph{a dual $t$-motive (or $t$-comotive \cite[Sec.~3.2]{GM25})} if $\mathcal{M}$ is free of finite rank over both $\oK[t]$ and $\oK[\sigma]$, and $(t-\theta)^n\mathcal{M}\subset\sigma\mathcal{M}$ for a sufficiently large integer $n$. Since $\mathcal{M}$ is free of finite rank over $\oK[t]$, if we set $r:=\rank_{\oK[t]}\mathcal{M}$ and fix a choice of a $\oK[t]$-basis $\bm\in\Mat_{r\times 1}(\mathcal{M})$, then there is a matrix $\Phi\in\Mat_{r}(\oK[t])\cap\GL_r(\oK(t))$ with $\det\Phi=c(t-\theta)^d$ for some $c\in\mathbb{C}_\infty^\times$ such that $\sigma\bm=\Phi\bm$. The dual $t$-motive $\mathcal{M}$ is called \emph{rigid analytically trivial} if there exists $\Psi\in\GL_r(\TT)$ such that $\Psi^{(-1)}=\Phi\Psi$. The pair $(\Phi,\Psi)$ is called a rigid analytic trivialization of $\mathcal{M}$. In fact, by \cite[Prop.~3.1.3]{ABP04}, the entries of $\Psi$ may be seen as entire functions of $t$. We denote by $\oK(\Psi(\theta))$ the field generated by entries of $\Psi(\theta)$ over $\oK$.
    
    For any $\sum_{i\geq 0}c_i\tau^i\in \oK[\tau]$, we let $(\sum_{i\geq 0}c_i\tau^i)^\star:=\sum_{i\geq 0}c_i^{(-i)}\sigma^i\in \oK[\sigma]$. Let $\phi$ be a Drinfeld module of rank $r$ defined over $\oK$. We associate to $\phi$ the set $\mathcal{M}_\phi:=\oK[\sigma]$. Note that $\mathcal{M}_{\phi}$ is a $\oK[t,\sigma]$-module. More precisely, it 
    has a natural left $\oK[\sigma]$-action and it is also equipped with the left $\oK[t]$-module structure, uniquely determined by
    \[
        t\cdot m:=m\phi_t^\star,~m\in\mathcal{M}_\phi.
    \]
We call $\mathcal{M}_{\phi}$ \emph{the dual $t$-motive of $\phi$}. Note that it is free of rank $r$ over $\oK[t]$. Moreover, it has a rigid analytic trivialization $(\Phi_{\phi},\Psi_{\phi})$ (see \cite[\S3.4]{CP12} for its explicit description).  By \cite[Prop. 3.4.7(c)]{CP12}, we have
    \begin{equation}\label{E:fieldperiod}
        \oK(P_\phi)=\oK(\Psi_{\phi}(\theta)).
    \end{equation}
Moreover, if we let
\[
\Omega(t):=(-\theta)^{-q/(q-1)}\prod_{i=1}^{\infty}\left(1-\frac{t}{\theta^{q^i}}\right)\in \mathbb{T}^{\times},
\]
then by \cite[Lem. 6.1.2]{GP19}, we see that 
\begin{equation}\label{E:det}
    \det(\Psi_{\phi})=\alpha \Omega(t)
\end{equation}
    for some $\alpha\in \oK^{\times}$.

    Before finishing this subsection, we note that these associated $t$-motivic objects will be fundamental for our later study in \S4 on algebraic relations among periods and quasi-periods of finitely many CM Drinfeld modules defined over $\oK$.

\subsection{Drinfeld modular forms}
 For $r\geq 2$, recall from \S1 the Drinfeld upper half plane $\Omega^r$ given by $\Omega^r:=\mathbb{P}^{r-1}(\mathbb{C}_\infty)\setminus\{K_\infty\text{-rational }\mathrm{hyperplanes}.\}$. After a certain normalization, we identify $\Omega^r$ with 
    \[
        \Omega^r=\{(w_1,\dots,w_{r-1},1)^\tr\in\Mat_{r\times 1}(\mathbb{C}_\infty)\mid w_1,\dots,w_{r-1},1~\mathrm{are}~K_\infty\text{-}\mathrm{linearly~independent}\}.
    \]
    Note that $\Omega^r$ is equipped with a structure of the rigid analytic space and each point $\boldsymbol{\omega}=(w_1,\dots,w_{r-1},1)^{\tr}$ corresponds to an $A$-lattice $\Lambda_{\bomega}$ of rank $r$ given by
    \[
\Lambda_{\boldsymbol{\omega}}=Aw_1+\cdots+Aw_{r-1}+A\subset\mathbb{C}_\infty.
    \]
    
In what follows, for each tuple $(b_1,\dots,b_r)\in K_{\infty}^r$ and $\bomega=(w_1,\dots,w_{r-1},1)^{\tr}$, we set $\ell_{b_1,\dots,b_r}(\bomega):=b_1w_1+\dots+b_{r-1}w_{r-1}+b_r$ and let $|\bomega|_{\infty}$ be the maximum among  $|w_1|,\dots,|w_{r-1}|$ and $1$.    
    For each positive integer $n$, we consider the subset $\Omega_{n}^r\subset \Omega^r$ which is the set of all $\bomega\in \Omega^r$ such that $|\ell_{b_1,\dots,b_r}(\bomega)|\geq q^{-n}|\bomega|_{\infty}$ for any $(b_1,\dots,b_r)\in K_{\infty}^r$. Then  $\{\Omega_{n}^r\}_{n=1}^{\infty}$ is an admissible covering of $\Omega^r$ \cite[\S3]{BBP21}. We call $f:\Omega^r\to \mathbb{C}_{\infty}$ \emph{a rigid analytic function} if its restriction to each $\Omega_n^r$ is the uniform limit of rational functions having no poles in $\Omega_n^r$.    
    
    Given $\gamma\in\GL_r(K_\infty)$ and $\boldsymbol{\omega}\in\Omega^r$, we set
    \[
        j(\gamma;\boldsymbol{\omega}):=\text{the last entry of}~\gamma\boldsymbol{\omega}.
    \]
    Then we can define $\GL_r(K_\infty)$-action on $\Omega^r$ by setting
    \[
        \gamma\cdot\boldsymbol{\omega}:=j(\gamma;\boldsymbol{\omega})^{-1}\gamma\boldsymbol{\omega}.
    \]
    For any rigid analytic function $f:\Omega^r\to\mathbb{C}_\infty$, $k\in\mathbb{Z}$ and $m\in\mathbb{Z}/(q-1)\mathbb{Z}$, we define
    \[
        (f\|_{k,m}\gamma)(\boldsymbol{\omega}):=\det(\gamma)^mj(\gamma;\boldsymbol{\omega})^{-k}f(\gamma\cdot\boldsymbol{\omega}).
    \]
    
    Let $N\in A$ be a monic polynomial. We set $\Gamma(1):=\GL_r(A)$ and for $N\in A\setminus \mathbb{F}_q$, define $\Gamma(N):=\Ker\big(\GL_r(A)\twoheadrightarrow\GL_r(A/N)\big)$ to be the kernel of the mod $N$ map. Note that $\Gamma(N)$ is a normal subgroup of $\Gamma(1)$ with finite index. By \emph{a congruence subgroup $\Gamma$ of $\GL_r(A)$}, we mean a subgroup of $\GL_r(A)$ that contains $\Gamma(N)$ for some monic polynomial $N$.
    \begin{definition} \textnormal{
        A rigid analytic function $f:\Omega^r\to\mathbb{C}_\infty$ is called \emph{a rank $r$  weak Drinfeld modular form of weight $k$ and type $m$ for a congruence subgroup $\Gamma$} if 
        \[
            (f\|_{k,m}\gamma)(\boldsymbol{\omega})=f(\boldsymbol{\omega}).
        \]
        We denote by $\mathcal{W}^r_{k,m}(\Gamma)$ the $\mathbb{C}_\infty$-vector space spanned by all rank $r$ weak Drinfeld modular forms of weight $k$ and type $m$ for the congruence subgroup $\Gamma$. We further set
        \[
            \mathcal{W}^r_k(\Gamma):=\bigcup_{m\in\mathbb{Z}/(q-1)\mathbb{Z}}\mathcal{W}^r_{k,m}(\Gamma)
        \]
        and
        \[
            \mathcal{W}^r(\Gamma):=\sum_{k\in\mathbb{Z}}\mathcal{W}^r_k(\Gamma)
        \]
        to be the $\mathbb{C}_\infty$-algebra generated by all rank $r$ weak Drinfeld modular forms for $\Gamma$.}
    \end{definition}

     Consider the unipotent embedding
    \begin{align*}
        \iota:A^{r-1}&\to\GL_r(A)\\
        (a_2,\dots,a_r)&\mapsto\begin{pmatrix}
            1 & a_2 & \cdots & a_r\\
             & \ddots & & \\
             & & \ddots & \\
             & & & 1
        \end{pmatrix}.
    \end{align*}
    For any congruence subgroup $\Gamma\subset\GL_r(A)$, we set $\Gamma_U:=\iota(A^{r-1})\cap\Gamma\subset\GL_r(A)$ and $\Lambda_\Gamma:=\iota^{-1}(\Gamma_U)\subset A^{r-1}$. For $\boldsymbol{\omega}=(w_1,\dots,w_{r-1},1)^{\tr}\in\Omega^r$, we let $\widetilde{\boldsymbol{\omega}}:=(w_2,\dots,w_{r-1},1)^{\tr}\in\Omega^{r-1}$. By a slight abuse of notation, we also write 
    \[
    \boldsymbol{\omega}=(w_1,w_2,\dots,w_{r-1},1)^{\tr}=(w_1,\widetilde{\boldsymbol{\omega}})^\tr\in \Omega^r.    
    \]
    We define 
    \[
        \Lambda_\Gamma\widetilde{\boldsymbol{\omega}}:=\{a_2w_2+\cdots+a_{r-1}w_{r-1}+a_r\in\mathbb{C}_\infty\mid(a_2,\dots,a_r)\in\Lambda_{\Gamma}\}
    \]
    to be an $A$-lattice of rank $r-1$ inside $\mathbb{C}_\infty$. Let $\widetilde{\pi}\in \mathbb{C}_{\infty}^{\times}$ be the Carlitz period defined in \S1. The parameter at infinity for $\Gamma$ is defined to be the function
    \begin{align*}
        u_\Gamma:\Omega^r&\to\mathbb{C}_\infty\\
        \boldsymbol{\omega}=(w_1,\widetilde{\boldsymbol{\omega}})^{\tr}&\mapsto u_{\Gamma}(\boldsymbol{\omega}):=\frac{1}{\exp_{\Lambda_{\Gamma}\widetilde{\pi}\widetilde{\boldsymbol{\omega}}}(\widetilde{\pi}w_1)}.
    \end{align*}
In particular, when $\Gamma=\Gamma(1)$, we have $u(\ww):=u_{\Gamma(1)}(\ww)=\exp_{A^{r-1}\widetilde{\pi}\ww}(\widetilde{\pi}w_1)^{-1}$ and, more generally, when $\Gamma=\Gamma(N)$ for some monic irreducible polynomial $N\in A$ of positive degree, we have 
\[
u_{N}(\ww):=u_{\Gamma(N)}(\ww)=\exp_{A^{r-1}\widetilde{\pi}\ww}\left(\frac{\widetilde{\pi}w_1}{N}\right)^{-1}.
\]
    
    A rigid analytic function $f:\Omega^r\to\mathbb{C}_\infty$ is called \emph{$\Gamma_U$-invariant} if for any $\boldsymbol{\omega}\in\Omega^r$ and $\gamma\in \Gamma_U$, we have
    \[
        f(\boldsymbol{\omega})=f(\gamma\cdot\boldsymbol{\omega}).
    \]
    It is known by \cite[Prop. 5.4]{BBP21} that any $\Gamma_U$-invariant rigid analytic function $f:\Omega^r\to\mathbb{C}_\infty$ admits a unique \emph{$u_\Gamma$-expansion}, namely for each $n\in\mathbb{Z}$, we have a unique rigid analytic function $f_n:\Omega^{r-1}\to\mathbb{C}_\infty$ so that
    \[
        f(\boldsymbol{\omega})=\sum_{n\in\mathbb{Z}}f_n(\widetilde{\boldsymbol{\omega}})u_\Gamma(\boldsymbol{\omega})^n
    \]
    whenever  $\boldsymbol{\omega}=(w_1,\widetilde{\boldsymbol{\omega}})^\tr\in\Omega^r$  lies in some neighborhood of infinity. By \cite[Prop. 3.2.7]{Bas14}, if $f\in\mathcal{W}^r(\Gamma)$, then, for each $n\geq 0$, we have $f_n\in\mathcal{W}^{r-1}(\widetilde{\Gamma})$ for some congruence subgroup $\widetilde{\Gamma}\in \GL_{r-1}(A)$. For $r=1$, we adopt the convention that $\Omega^1:=\{1\}$ and $\mathcal{O}(\Omega^1):=\mathbb{C}_\infty$ and for $r\geq 2$, we let $\mathcal{O}(\Omega^r)$ be the ring of $\mathbb{C}_\infty$-valued rigid analytic functions on $\Omega^r$. Then $u_\Gamma$-expansion induces an injection
    \begin{equation}\label{E:injection}
    \begin{split}
        \{f\in\mathcal{O}(\Omega^r)\mid f~\text{is}~\Gamma_U\text{-invariant}\}&\to\mathcal{O}(\Omega^{r-1})\llbracket X,X^{-1}\rrbracket\\
        f&\mapsto\mathscr{F}_f:=\sum_{n\in\mathbb{Z}}f_nX^n.
    \end{split}
    \end{equation}
Given $f\in\mathcal{O}(\Omega^r)$ that is $\Gamma_U$-invariant, we say that $f$ is \emph{holomorphic at infinity with respect to $\Gamma$} if $\mathscr{F}_f\in\mathcal{O}(\Omega^{r-1})\llbracket X\rrbracket$. Observe that every weak Drinfeld modular form for $\Gamma$ is automatically $\Gamma_U$-invariant.
    \begin{definition}\textnormal{
        A rank $r$ weak Drinfeld modular form $f$ of weight $k$ and type $m$ for a congruence subgroup $\Gamma$ is called \emph{a rank $r$ Drinfeld modular form of weight $k$ and type $m$ for $\Gamma$} if for each $\alpha\in\Gamma(1)$, $f\|_{k,m}\alpha$ is holomorphic at infinity with respect to $\alpha^{-1}\Gamma\alpha$. We denote by $\mathcal{M}^r_{k,m}(\Gamma)$ the $\mathbb{C}_\infty$-vector space spanned by all rank $r$ Drinfeld modular forms of weight $k$ and type $m$ for the congruence subgroup $\Gamma$. We also let
        \[
            \mathcal{M}^r_k(\Gamma):=\bigcup_{m\in\mathbb{Z}/(q-1)\mathbb{Z}}\mathcal{M}^r_{k,m}(\Gamma)
        \]
        and set
        \[
            \mathcal{M}^r(\Gamma):=\sum_{k\in\mathbb{Z}}\mathcal{M}^r_k(\Gamma)
        \]
         to be the $\mathbb{C}_\infty$-algebra generated by all rank $r$ Drinfeld modular forms for $\Gamma$. Since, for every monic polynomial $N$ of positive degree and $\gamma \in \Gamma(N)$, we have $\det(\gamma)=1$, we simply write
         \[
         \mathcal{M}^r_{k}(\Gamma(N)):=\mathcal{M}^r_{k,0}(\Gamma(N)).
         \]}
    \end{definition}

We finish this subsection by providing several examples of Drinfeld modular forms.
\begin{example}\label{Ex:1}
\begin{itemize}
    \item[(i)]  \textnormal{Given $\boldsymbol{\omega}=(w_1,\dots,w_{r-1},1)^{\tr}\in\Omega^r$, we denote by $\phi^{\bomega}$ the Drinfeld module corresponding to the $A$-lattice 
    $        \Lambda_{\boldsymbol{\omega}}.
    $
     Given $a\in\mathbb{F}_q[t]$, we have
    \[
        \phi^{\boldsymbol{\omega}}_a(X)=a(\theta)X+g_{1,a}(\boldsymbol{\omega})X^q+\cdots+g_{r\deg(a),a}(\boldsymbol{\omega})X^{q^{r\deg(a)}}.
    \]
   For each $1\leq i\leq r\deg(a)$, the function $g_{i,a}:\Omega^r\to\mathbb{C}_\infty$ is a rank $r$ Drinfeld modular form of weight $q^i-1$ and type $0$ for $\Gamma(1)$. For each $1\leq \mu \leq d\deg(a) $, we call each $g_{\mu,a}$ \emph{a coefficient form}.  When $a=t$, we set $g_i:=g_{i,t}$ for any $1\leq i\leq r$. In fact, by \cite[(3.3), (3.4) and (3.5)]{BB17}, for any $a\in\mathbb{F}_q[t]$ and $1\leq i\leq r\deg(a)$, we have
    \[
        g_{i,a}\in A[g_1,\dots,g_r].
    \]
}
    \item[(ii)] \textnormal{Let $N$ be a monic polynomial in $A$. We set $\mathcal{T}_N:=(N^{-1}A/A)^r\setminus\{0\}$. For each $\mathbf{u}=(u_1,\dots,u_r)\in\mathcal{T}_N$ and $\boldsymbol{\omega}=(w_1,\dots,w_{r-1},1)^{\tr}\in\Omega^r$, we write
\[
(\boldsymbol{a}+\bu)\cdot\boldsymbol{\omega}:=(a_1+u_1)w_1+\cdots+(a_{r-1}+u_{r-1})w_{r-1}+(a_r+u_r).
\]
    Then we define \emph{the weight one Eisenstein series for $\Gamma(N)$} by $E_{\mathbf{u},N}:\Omega^r\to\mathbb{C}_\infty$ by
    \[
        E_{\mathbf{u},N}(\boldsymbol{\omega}):=\sum_{\boldsymbol{a}=(a_1,\dots,a_r)\in A^r}\frac{1}{(\boldsymbol{a}+\bu)\cdot\boldsymbol{\omega}}.
    \]
    Note that $E_{\mathbf{u},N}$ is a nowhere vanishing Drinfeld modular form of weight one for $\Gamma(N)$ and hence lies in $\mathcal{M}^r_{1}(\Gamma(N))$. For simplicity, when $N=\theta$, we further write $E_{\mathbf{u}}:=E_{\mathbf{u},\theta}$}. 
\end{itemize}

\end{example}

We finish this section with a crucial lemma that will be later used in \S3 to obtain Theorem \ref{Thm:Special_Values_to_Periods}. We note that Lemma \ref{L:integral} can be also deduced from \cite[Thm. 5.4, (7.20)]{Gek25}. For the convenience of the reader, we provide the details of its proof.

\begin{lemma}[{cf. \cite[Thm. 5.4]{Gek25}}]\label{L:integral} Let $f\in \mathcal{M}_k^r(\Gamma(N))$ for $k\in \mathbb{Z}_{\geq 0}$ and some monic polynomial $N\in A$ of positive degree. Then $f$ is integral over $\mathbb{C}_{\infty}[g_1,\dots,g_r]$. In particular, $f$ is integral over the $\mathbb{C}_{\infty}$-algebra generated by all the Eisenstein series $E_{\mathbf{u}}$ where $\mathbf{u}\in\mathcal{T}_{\theta}$.
\end{lemma}
\begin{proof} We start with noting that, by \cite[Thm. 17.5(a)]{BBP21}, the $\mathbb{C}_{\infty}$-algebra of Drinfeld modular forms of type $0$ is generated by $g_1,\dots,g_r$ over $\mathbb{C}_{\infty}$. Now, for any $\ell\in \mathbb{Z}_{\geq 0}$, letting $(\mathcal{M}_{\ell}^r(\Gamma(N)))^{\Gamma(1)}$ be the $\mathbb{C}_{\infty}$-vector space of elements in $\mathcal{M}_{\ell}^r(\Gamma(N))$ invariant under the operator $\|_{\ell,0}\gamma$ for any $\gamma \in \Gamma(1)$, by \cite[(6.7)]{BBP21}, we see that 
\begin{equation}\label{E:inv}
    (\mathcal{M}_{\ell}^r(\Gamma(N)))^{\Gamma(1)}=\mathcal{M}_{\ell,0}^r(\Gamma(1)).
    \end{equation}
Recall that $\Gamma(N)$ is of finite index in $\Gamma(1)$, say $\mathfrak{d}:=|\Gamma(1)/\Gamma(N)|$. We consider the polynomial
\[
\mathfrak{p}_f(X):=\prod_{\delta\in \Gamma(1)/\Gamma(N)}(X-f\|_{k,0}\delta)\in \mathcal{O}(\Omega^r)[X].
\]
    For each $0\leq i \leq \mathfrak{d}$, let $\mathfrak{p}_i$ be the $i$-th coefficient of $\mathfrak{p}_f$. Then $\mathfrak{p}_{\mathfrak{d}}=1$. Note that since $\Gamma(N)$ is a normal subgroup of $\Gamma(1)=\GL_r(A)$, for any $\gamma\in \Gamma(1)$, we have $\gamma^{-1}\Gamma(N)\gamma=\Gamma(N)$ and hence by \cite[Prop. 6.6]{BBP21}, we obtain $f\|_{k,0}\gamma\in \mathcal{M}_k^r(\Gamma(N))$. In particular, for $i<\mathfrak{d}$, we have $\mathfrak{p}_i\in \mathcal{M}^{r}_{k(\mathfrak{d}-i)}(\Gamma(N))$. On the other hand, by \cite[(1.6)]{BBP21} and the construction of $\mathfrak{p}_i$, we see that $
    \mathfrak{p}_i\|_{k(\mathfrak{d}-i),0}\gamma= \mathfrak{p}_i.$
    Hence by \eqref{E:inv}, we have $\mathfrak{p}_i\in \mathcal{M}_{k(\mathfrak{d}-i),0}^r(\Gamma(1))$. Since $\mathfrak{p}_f$ is monic and $\mathfrak{p}_f(f)=0$, $f$ is integral over $\mathbb{C}_{\infty}[g_1,\dots,g_r]$. The second assertion follows from the first assertion and \cite[(3.5), (3.6)]{Gek17}.
\end{proof}
    
\section{Arithmetic Drinfeld modular forms}
    
    The goal of this section is to show that the special values of arithmetic Drinfeld modular forms at CM points are given by algebraic multiple of periods of CM Drinfeld modules defined over $\oK$.  
    
    Let $\Gamma$ be a congruence subgroup of $\Gamma(1)$. We first set $\mathcal{AW}^1(\Gamma):=\oK$. Let $r\geq 2$. We call $f\in \mathcal{W}^r(\Gamma)$  \emph{an arithmetic weak Drinfeld modular form} if $ \mathscr{F}_f\in\mathcal{AW}^{r-1}(\Gamma)\llbracket X,X^{-1}\rrbracket $. We denote by $\mathcal{AW}^r(\Gamma)$ the $\oK$-algebra of arithmetic weak Drinfeld modular forms. 
    
    By \cite[Prop. 3.2.7]{Bas14}, note that if $f\in\mathcal{W}^r(\Gamma)$, then we have $\mathscr{F}_f\in\mathcal{W}^{r-1}(\Gamma)\llbracket X,X^{-1}\rrbracket$. Furthermore, we call $f\in \mathcal{M}^r(\Gamma)$ \emph{an arithmetic Drinfeld modular form} if it lies in $\mathcal{AW}^r(\Gamma)$. We denote by  
    \[
        \mathcal{AM}^r(\Gamma):=\mathcal{M}^r(\Gamma)\cap\mathcal{AW}^{r}(\Gamma)
    \]
    the $\oK$-algebra of the rank $r$ arithmetic Drinfeld modular forms for $\Gamma$. 
    Observe that every element $f\in\mathcal{AM}^r(\Gamma)$ must satisfy $\mathscr{F}_f\in\mathcal{AW}^{r-1}(\Gamma)\llbracket X\rrbracket$.

\begin{example}\label{Ex:2} \textnormal{Using an appropriate normalization, all coefficient forms are arithmetic. More precisely, we define
    \[
        g^{\mathrm{ari}}_{i,a}:=\widetilde{\pi}^{1-q^{i}}g_{i,a}.
    \]
    Then by \cite[Thm.~1.1]{Sug18}, we have
    \[
        g^{\mathrm{ari}}_{i,a}\in\mathcal{AM}^r\big(\Gamma(1)\big).
    \]}
\end{example}
    Now we extend the notion of arithmeticity to meromorphic functions on $\Omega^r$. For each $k\in\mathbb{Z}$, we set
    \[
        \mathcal{F}^r_{k}(\Gamma):=\mathrm{Span}_{\mathbb{C}_\infty}\left(f/g\mid f\in\mathcal{M}^r_{k+\mu}(\Gamma),~g\in\mathcal{M}^r_{\mu}(\Gamma),~\mu\geq 1\right)
    \]
    to be the $\mathbb{C}_\infty$-vector space spanned by rank $r$ meromorphic Drinfeld modular forms of weight $k$. The elements of $\mathcal{F}^r_0(\Gamma)$ are called \emph{Drinfeld modular functions for $\Gamma$}. We further denote by $\mathcal{AF}^r_k(\Gamma)$ the $\mathbb{C}_\infty$-vector space spanned by \emph{arithmetic meromorphic Drinfeld modular forms for $\Gamma$} which is given by
    \[
        \mathcal{AF}^r_k(\Gamma):=\mathrm{Span}_{\mathbb{C}_\infty}\left(f/g\mid f\in\mathcal{AM}^r_{k+\ell}(\Gamma),~g\in\mathcal{AM}^r_{\mu}(\Gamma),~\mu\geq 1\right).
    \]

\subsection{Arithmetic Drinfeld modular forms with level $N$} As before, let $N\in A$ be a monic polynomial.  The goal of this subsection is to study the $\oK$-algebra of arithmetic Drinfeld modular forms with level $N$. We begin with analyzing the $u_{\Gamma(N)}$-expansion $\mathscr{F}_{E_{\mathbf{u},N}}$ of $E_{\mathbf{u},N}$ for any $\mathbf{u}\in \mathcal{T}_N$. For $0\neq a\in\mathbb{F}_q[t]$ and $\widetilde{\boldsymbol{\omega}}\in\Omega^{r-1}$, we set $d(a):=(r-1)\deg(a)$ and define the reciprocal polynomial
    \begin{align*}
        f^{\widetilde{\boldsymbol{\omega}}}_a(X)&:=X^{q^{d(a)}}\widetilde{g}_{d(a),a}(\widetilde{\boldsymbol{\omega}})^{-1}\phi^{\Tilde{\boldsymbol{\omega}}}_a(X^{-1})\\
        &=1+\frac{\widetilde{g}_{d(a)-1,a}(\widetilde{\boldsymbol{\omega}})}{\widetilde{g}_{d(a),a}(\widetilde{\boldsymbol{\omega}})}X^{q^{d(a)}-q^{d(a)-1}}+\cdots+\frac{a(\theta)}{\widetilde{g}_{d(a),a}(\widetilde{\boldsymbol{\omega}})}X^{q^{d(a)}-1},
    \end{align*}
    where $\widetilde{g}_{i,a}:\Omega^{r-1}\to\mathbb{C}_\infty$ is a coefficient form sending each $\tomega\in \Omega^{r-1}$ to the $i$-th coefficient of $\phi_a^{\tomega}$. It induces a map
    \begin{align*}
        f^{(\cdot)}_a:\Omega^{r-1}&\to 1+X\mathbb{C}_\infty[X]\\
        \widetilde{\boldsymbol{\omega}}&\mapsto f^{\widetilde{\boldsymbol{\omega}}}_a(X).
    \end{align*}
    
    Let us denote by $\phi^{\widetilde{\pi}\widetilde{\boldsymbol{\omega}}}$ the Drinfeld module corresponding to the lattice $\widetilde{\pi}\Lambda_{\widetilde{\boldsymbol{\omega}}}$ generated by the elements $\tilde{\pi}w_2,\dots,\tilde{\pi}w_{r-1},\tilde{\pi}$ over $A$. Using the relation (see \cite[\S4]{Gek89})
    \[
    \exp_{\phi^{\widetilde{\pi}\tomega}}(X)=\widetilde{\pi}\exp_{\phi^{\tomega}}(X)
    \]
    as well as the functional equation \eqref{E:exp}, we see that
    \begin{align*}
        \phi^{\widetilde{\pi}\widetilde{\boldsymbol{\omega}}}_a(X)&=\widetilde{\pi}\phi_a^{\widetilde{\boldsymbol{\omega}}}(\widetilde{\pi}^{-1}X)=a(\theta)X+\widetilde{g}_{1,a}^{\mathrm{ari}}(\widetilde{\boldsymbol{\omega}})X^q+\cdots+\widetilde{g}_{d(a),a}^{\mathrm{ari}}(\widetilde{\boldsymbol{\omega}})X^{q^{d(a)}}.
    \end{align*}
    We further define its reciprocal polynomial
    \begin{align*}
        f^{\widetilde{\pi}\widetilde{\boldsymbol{\omega}}}_a(X)&:=X^{q^{d(a)}}\widetilde{g}_{d(a),a}^{\mathrm{ari}}(\widetilde{\boldsymbol{\omega}})^{-1}\phi^{\widetilde{\pi}\Tilde{\boldsymbol{\omega}}}_a(X^{-1})\\
        &=f^{\widetilde{\boldsymbol{\omega}}}_a(\widetilde{\pi}X)\\
        &= 1+\frac{\widetilde{g}_{d(a)-1,a}^{\mathrm{ari}}(\widetilde{\boldsymbol{\omega}})}{\widetilde{g}_{d(a),a}^{\mathrm{ari}}(\widetilde{\boldsymbol{\omega}})}X^{q^{d(a)}-q^{d(a)-1}}+\cdots+\frac{a(\theta)}{\widetilde{g}_{d(a),a}^{\mathrm{ari}}(\widetilde{\boldsymbol{\omega}})}X^{q^{d(a)}-1}.
    \end{align*}
    It also induces a map
    \begin{align*}
        f^{\widetilde{\pi}(\cdot)}_a:\Omega^{r-1}&\to 1+\mathbb{C}_\infty[X]X\\
        \widetilde{\boldsymbol{\omega}}&\mapsto f^{\widetilde{\pi}\widetilde{\boldsymbol{\omega}}}_a(X).
    \end{align*}
    
In what follows, when $r\geq 3$, to distinguish certain congruence subgroups of $\GL_{r-1}(A)$ and $\GL_r(A)$, for any monic polynomial $N$ in $A$, we further let $\widetilde{\Gamma}(N)\subset \GL_{r-1}(A)$ be the subgroup of elements in the kernel of the mod $N$ map. 

Our next result is a finer description of $f^{\widetilde{\pi}(\cdot)}_a$. For the convenience of the reader, for $1\leq i \leq r-1$, we recall that 
    \[
\widetilde{g}^{\mathrm{ari}}_{i,a}:=\widetilde{\pi}^{1-q^{i}}\widetilde{g}_{i,a}\in \mathcal{AM}^{r-1}(\widetilde{\Gamma}(1))
    \]  
    and
    \[
\widetilde{g}^{\mathrm{ari}}_{i}:=\widetilde{\pi}^{1-q^{i}}\widetilde{g}_{i,t}\in \mathcal{AM}^{r-1}(\widetilde{\Gamma}(1)).
    \]
    \begin{lemma}\label{L:recpoly}
        For $0\neq a\in\mathbb{F}_q[t]$, we have
        \[
            f^{\widetilde{\pi}(\cdot)}_a\in 1+A[\widetilde{g}_1^{\mathrm{ari}},\dots,\widetilde{g}_{r-2}^{\mathrm{ari}},(\widetilde{g}_{r-1}^{\mathrm{ari}})^{\pm 1}][X]X\subset 1+\mathcal{AW}^{r-1}(\widetilde{\Gamma}(1))[X]X.
        \]
    \end{lemma}
    
    \begin{proof}
        Assume that $a=a_0+a_1t+\cdots+a_st^s$ for some $a_j\in\mathbb{F}_q$ with $a_s\neq 0$. Then we must have 
        \[
            \widetilde{g}_{d(a),a}^{\mathrm{ari}}=a_s(\widetilde{g}_{r-1}^{\mathrm{ari}})^{\frac{1-q^{d(a)}}{1-q^{r-1}}}.
        \]
        In particular, we obtain
        \begin{equation}\label{Eq:inverse_of_leading_term}
            (\widetilde{g}_{d(a),a}^{\mathrm{ari}})^{-1}\in A[(\widetilde{g}_{r-1}^{\mathrm{ari}})^{-1}]\subset A[\widetilde{g}_1^{\mathrm{ari}},\dots,\widetilde{g}_{r-2}^{\mathrm{ari}},(\widetilde{g}_{r-1}^{\mathrm{ari}})^{\pm 1}].
        \end{equation}
        The desired result now follows from \eqref{Eq:inverse_of_leading_term} and the fact that $\widetilde{g}_{i,a}^{\mathrm{ari}}\in A[\widetilde{g}_1^{\mathrm{ari}},\dots,\widetilde{g}_{r-1}^{\mathrm{ari}}]$ for each $1\leq i \leq r-1$, which follows from \cite[(3.3), (3.4) and (3.5)]{BB17} after normalization.
    \end{proof}

    The next lemma is to obtain compatibility of the arithmetic property of weak Drinfeld modular forms with the level change.

    \begin{lemma}\label{Lem:Level_Change}  Let $(N_1)$ and $(N_2)$ be principal ideals in $A$ generated by the monic polynomials $N_1$ and $N_2$ respectively. If $(N_1)\subseteq (N_2)$, then
    \[
    \mathcal{AW}^r(\Gamma(N_2))\subseteq\mathcal{AW}^r(\Gamma(N_1)).
    \]
        In particular, for any $0\neq a\in\mathbb{F}_q[t]$, we have
        \[
            f^{\widetilde{\pi}(\cdot)}_a\in 1+\mathcal{AW}^{r-1}\big(\widetilde{\Gamma}(N)\big)[X]X.
        \]
    \end{lemma}

    \begin{proof} Since any element of $\mathcal{AW}^{r}(\Gamma(N_1))$ ($ \mathcal{AW}^{r}(\Gamma(N_2)) $ resp.) admits a unique $u_{N_1}$ ($u_{N_2}$ resp.) expansion, to prove above statements, it suffices to show that $u_{N_2}$ may be written as a power series in $u_{N_1}$ whose coefficients are elements in $\mathcal{AW}^{r-1}(\widetilde{\Gamma})$ for some congruence subgroup $\widetilde{\Gamma}$ of $\GL_{r-1}(A)$.  
    
    Since $N_2|N_1$, there exists $\mathfrak{n}\in A$ such that $\mathfrak{n}N_2=N_1$. 
       By the functional equation \eqref{E:exp}, for $\bomega=(w_1,\tomega)^{\tr}\in \Omega^r$, we have
       \begin{equation}\label{E:expofu}
        \begin{split}
            u_{N_2}(\boldsymbol{\omega})&=\frac{1}{\exp_{A^{r-1}\widetilde{\pi}\tomega}\left(\frac{\widetilde{\pi}w_1}{N_2}\right)}=\frac{1}{\phi^{\widetilde{\pi}\tomega}_{\mathfrak{n}}\left(\exp_{A^{r-1}\widetilde{\pi}\tomega}\left(\frac{\widetilde{\pi}w_1}{N_1}\right)\right)}=\frac{u_{N_1}(\boldsymbol{\omega})^{q^{d(N_1)}}}{\widetilde{g}_{d(\mathfrak{n}),\mathfrak{n}}^{\mathrm{ari}}(\tomega)f_{\mathfrak{n}}^{\widetilde{\pi}\tomega}\big(u_{N_1}(\boldsymbol{\omega})\big)}.
        \end{split}
        \end{equation}
  We will proceed by the induction on $r$. When $r=2$, since $\phi^{\widetilde{\pi}\tomega}$ is the Carlitz module, defined over $A$, and in this case, $ \widetilde{g}_{d(\mathfrak{n}),\mathfrak{n}}^{\mathrm{ari}}(\tomega)=1$, the lemma follows. Assume that the desired assertion holds for $r-1$, that is,
  \[
    \mathcal{AW}^{r-1}(\widetilde{\Gamma}(N_2))\subseteq\mathcal{AW}^{r-1}(\widetilde{\Gamma}(N_1))
    \]  
  provided that $(N_1)\subseteq (N_2)$. By \cite[Prop. 1.3]{Sug18}, we know that $ (\widetilde{g}_{d(\mathfrak{n}),\mathfrak{n}}^{\mathrm{ari}})^{-1} \in \mathcal{AW}^{r-1}(\widetilde{\Gamma}(1))$ and hence, by our induction argument, we have
  \[
(\widetilde{g}_{d(\mathfrak{n}),\mathfrak{n}}^{\mathrm{ari}})^{-1}\in\mathcal{AW}^{r-1}\big(\widetilde{\Gamma}(1)\big)\subseteq \mathcal{AW}^{r-1}\big(\widetilde{\Gamma}(N_2)\big) \subseteq \mathcal{AW}^{r-1}\big(\widetilde{\Gamma}(N_1)\big). 
\]
Moreover, again by our induction argument and by Lemma \ref{L:recpoly}, we have   
\[
f_{\mathfrak{n}}^{\widetilde{\pi}\tomega} \in 1 +\mathcal{AW}^{r-1}\big(\widetilde{\Gamma}(N_1)\big)[X]X.
\]
 Since,  in the expansion of  of $\frac{1}{f_{\mathfrak{n}}^{\widetilde{\pi}\tomega}} $ as a power series in $X$, only finitely many elements in $\mathcal{AW}^{r-1}(\widetilde{\Gamma}(N_1))$ contributes to the each coefficient of $X$, by \eqref{E:expofu}, one obtains
        \[
            u_{N_2}(\boldsymbol{\omega})\in\mathcal{AW}^{r-1}\big(\widetilde{\Gamma}(N_1)\big)[u_{N_1}(\boldsymbol{\omega})]u_{N_1}(\boldsymbol{\omega})^{q^{d(N_1)}}.
        \]
        The desired result now follows.
    \end{proof}
We now set 
\[
E_{\mathbf{u},N}^{\mathrm{ari}}:=\widetilde{\pi}^{-1}E_{\mathbf{u},N}\in \mathcal{M}^r(\Gamma(N)).
\]    
and 
\[
E_{\mathbf{u}}^{\mathrm{ari}}:=\widetilde{\pi}^{-1}E_{\mathbf{u}}\in \mathcal{M}^r(\Gamma(\theta)).
\]
\begin{lemma}\label{Lem:Expansion_of_Eisenstein_Series} 
        For each $\mathbf{u}\in\mathcal{T}_N$, we have
        \[
            \widetilde{\pi}^{-1}\mathscr{F}_{E_\mathbf{u}, N}\in\mathcal{AW}^{r-1}\big(\Gamma(N)\big)\llbracket X\rrbracket.
        \]
        In particular, $E_{\mathbf{u},N}^{\mathrm{ari}}\in\mathcal{AM}^r\big(\Gamma(N)\big)$ and $E_{\mathbf{u}}^{\mathrm{ari}}\in\mathcal{AM}^r\big(\Gamma(\theta)\big)$.
    \end{lemma}

    
    \begin{proof}
        We will proceed the proof by using the induction on the rank $r$. The initial case $r=2$ is known by \cite[(2.1)]{Gek84} for any monic polynomial $N$. Now we assume that for each $\widetilde{\bu}\in(N^{-1}A/A)^{r-1}$, we have
        \begin{equation}\label{Eq:Induction_Hypothesis}
            E_{\widetilde{\mathbf{u}},N}^{\mathrm{ari}}:=\widetilde{\pi}^{-1}E_{\widetilde{\mathbf{u}},N}\in\mathcal{AM}^{r-1}\big(\widetilde{\Gamma}(N)\big).
        \end{equation}
        Taking the logarithmic derivative of both sides of
        \[
            \exp_{A^{r-1}\widetilde{\pi}\tomega}(z)=z\prod_{\lambda \in A^{r-1}\widetilde{\pi}\tomega} \left(1-\frac{z}{\lambda}\right),
        \] 
        we obtain 
        \begin{equation}\label{E:logder}
            \frac{1}{\exp_{A^{r-1}\widetilde{\pi}\tomega}(z)}=\sum_{b\in A^{r-1}\widetilde{\pi}\tomega}\frac{1}{z-b}.
        \end{equation}
        Let us denote each representative $\mathbf{u}\in\mathcal{T}_N$ by $\mathbf{u}=(u_1,\tilde{\mathbf{u}})=(\frac{v_1}{N},\tilde{\mathbf{u}})$ for some $v_1\in A$ with $\deg(v_1)<\deg(N)$. Then we have 
        \begin{align*}
            \widetilde{\pi}^{-1}E_{\mathbf{u},N}(\bomega)&=\sum_{(a_1,\dots,a_r)\in A^{r}}\frac{1}{\widetilde{\pi}(a_1+u_1)w_1+\cdots+\widetilde{\pi}(a_{r-1}+u_{r-1})w_{r-1}+\widetilde{\pi}(a_r+u_r)}\\
            &= \sum_{a_1\in A} \sum_{\tilde{a}=(a_2,\dots,a_r)\in A^{r-1}} \frac{1}{\widetilde{\pi}(a_1+u_1)w_1+\widetilde{\pi}\tomega\cdot \tilde{\mathbf{u}}+\widetilde{\pi}\tomega\cdot \tilde{a}}\\
            &= \sum_{a_1\in A}\frac{1}{\exp_{A^{r-1}\widetilde{\pi}\tomega}(\widetilde{\pi}(a_1+u_1)w_1+\widetilde{\pi}\tomega\cdot \tilde{\mathbf{u}})}\\
            &=\sum_{a_1\in A}\frac{1}{\exp_{A^{r-1}\widetilde{\pi}\tomega}\left(\widetilde{\pi}\left(a_1+\frac{v_1}{N}\right) w_1\right)+ \exp_{A^{r-1}\widetilde{\pi}\tomega}(\widetilde{\pi}\tomega\cdot \tilde{\mathbf{u}})}\\
            &=\sum_{a_1\in A}\frac{1}{\phi^{\widetilde{\pi}\tomega}_{a_1N+v_1}(u_{N}^{-1})+ \exp_{A^{r-1}\widetilde{\pi}\tomega}(\widetilde{\pi}\tomega\cdot \tilde{\mathbf{u}})}\\
            &=\sum_{a_1\in A}\frac{1}{\phi^{\widetilde{\pi}\tomega}_{a_1N+v_1}(u_{N}^{-1})+E_{\widetilde{\bu},N}^{\mathrm{ari}}(\tomega)}\\
            &=\sum_{a_1\in A}\frac{1}{f_{a_1N+v_1}^{\widetilde{\pi}\tomega}(u_N)\widetilde{g}_{d(a_1N+v_1),a_1N+v_1}^{\mathrm{ari}}u_N^{-q^{d(a_1N+v_1)}}+E_{\widetilde{\bu},N}^{\mathrm{ari}}(\tomega)}\\
            &=\sum_{a_1\in A}\frac{1}{1+F_{a_1}(\boldsymbol{\omega})}\frac{u_N^{q^{d(a_1N+v_1)}}}{\widetilde{g}_{d(a_1N+v_1),a_1N+v_1}^{\mathrm{ari}}}\\
            &=\sum_{a_1\in A}\frac{u_N^{q^{d(a_1N+v_1)}}}{\widetilde{g}_{d(a_1N+v_1),a_1N+v_1}^{\mathrm{ari}}}\sum_{i\geq 0}(-\mathcal{P}_{a_1}(\boldsymbol{\omega}))^i,
        \end{align*}
        where
        \begin{align*}
                \mathcal{P}_{a_1}(\bomega)=(f_{a_1N+v_1}^{\widetilde{\pi}\tomega}(u_N)-1)+\frac{E_{\widetilde{\bu},N}^{\mathrm{ari}}(\tomega)}{\widetilde{g}_{d(a_1N+v_1),a}^{\mathrm{ari}}}u_N^{q^{d(a_1N+v_1)}}.
        \end{align*}
        Here, the third equality follows from \eqref{E:logder} and the last equality follows from the functional equation of $\exp_{A^{r-1}\widetilde{\pi}\tomega}$. One the one hand, using Lemma~\ref{Lem:Level_Change}, we have
        \[
            f_{a_1N+v_1}^{\widetilde{\pi}\tomega}(u_N)-1\in\mathcal{AW}^{r-1}\big(\widetilde{\Gamma}(N)\big)[u_N]u_N.
        \]
        On the other hand, by \eqref{Eq:inverse_of_leading_term}, Lemma~\ref{Lem:Level_Change}, and the induction hypothesis applied to $E_{\widetilde{\mathbf{u}},N}^{\mathrm{ari}}$, we have
        \[
            \frac{E_{\widetilde{\bu},N}^{\mathrm{ari}}(\tomega)}{\widetilde{g}_{d(a_1N+v_1),a}^{\mathrm{ari}}}u_N^{q^{d(a_1N+v_1)}}\in\mathcal{AW}^{r-1}\big(\widetilde{\Gamma}(N)\big)[u_N]u_N.
        \]
        It follows that
        \[
            \mathcal{P}_{a_1}(\bomega)\in\mathcal{AW}^{r-1}\big(\widetilde{\Gamma}(N)\big)[u_N]u_N.
        \]
        Since
        \[
            \mathrm{ord}_{u_N}\bigg(\frac{u_N^{q^{d(a_1N+v_1)}}}{\widetilde{g}_{d(a_1N+v_1),a_1N+v_1}^{\mathrm{ari}}}\sum_{i\geq 0}(-F_{a_1}(\boldsymbol{\omega}))^i\bigg)\geq q^{d(a_1N+v_1)},
        \]
        if we fix $\ell\geq 0$, then there are finitely many $a_1$ whose summand contribute to the coefficient of $u_N^\ell$. Therefore, we conclude that 
        \[
            \widetilde{\pi}^{-1}\mathscr{F}_{E_\mathbf{u},N}\in\mathcal{AW}^{r-1}\big(\widetilde{\Gamma}(N)\big)\llbracket X\rrbracket.
        \]
        The second assertion is a consequence of the first assertion.
    \end{proof}
    
\subsection{Special CM values of arithmetic Drinfeld modular forms}
    In what follows, we aim to show the connection between the values of the arithmetic Drinfeld modular forms at CM points and the periods of certain CM Drinfeld modules. We begin with the result concerning the special CM values of the ratio of two weight one Eisenstein series. As usual, let $N$ be a monic polynomial in $A$. Since  these ratios generate the function field of the $\mathbb{C}_\infty$-varieties $\Gamma(N)\backslash\Omega^r$ \cite[Prop.~2.6]{Gek19}, our result can be seen as a generalization of Hamahata's result on special CM values of J-invariants \cite[Prop.~6.2]{Ham03}. Recall that  $\boldsymbol{\omega}\in \Omega^r$ is a CM point if $$\dim_K\End(\Lambda_{\boldsymbol{\omega}})\otimes_A K=r.$$

    \begin{lemma}\label{Lem:CM_values_E}
        Let $\bu\in\mathcal{T}_N$ and $\boldsymbol{\omega}\in\Omega^r$ be a CM point. Then there exists a period $\lambda_{\boldsymbol{\omega}}$ of a CM Drinfeld module defined over $\oK$ such that
        \[
            E_{\bu,N}(\boldsymbol{\omega})\in\oK\lambda_{\boldsymbol{\omega}}.
        \]
        In particular, for $\bv_1,\bv_2\in\mathcal{T}_N$, we have
        \[
            \frac{E_{\bv_1,N}(\boldsymbol{\omega})}{E_{\bv_2,N}(\boldsymbol{\omega})}\in\oK.
        \]
    \end{lemma}

    \begin{proof} 
        Since $\bomega$ is a CM point, $\phi^{\bomega}$ is a CM Drinfeld module. It follows from class field theory over global function fields \cite[Thm.~8.5]{Hay79} and \cite[Thm.~A.1]{Wei20} that there exists a Drinfeld module $\psi^{\boldsymbol{\omega}}$ defined over $H_{\bomega}\subset \overline{K}$, the ring class field of $\End(\Lambda_{\bomega})$, such that $\phi^{\boldsymbol{\omega}}\cong\psi^{\boldsymbol{\omega}}$. If we denote by $\Lambda_\psi$ the period lattice of $\psi^{\boldsymbol{\omega}}$, then there exists  $\lambda_{\boldsymbol{\omega}}\in\mathbb{C}_\infty^\times$ so that $\Lambda_\psi=\lambda_{\boldsymbol{\omega}}\Lambda_{\boldsymbol{\omega}}$. Moreover, since $1\in\Lambda_{\boldsymbol{\omega}}$, we must have $\lambda_{\boldsymbol{\omega}}\in\Lambda_\psi$. Thus,
        \[
            E_{\bu,N}(\boldsymbol{\omega})=\exp_{\Lambda_{\boldsymbol{\omega}}}(\bu\cdot\boldsymbol{\omega})^{-1}=\exp_{\lambda_{\boldsymbol{\omega}}^{-1}\lambda_{\boldsymbol{\omega}}\Lambda_{\boldsymbol{\omega}}}(\bu\cdot\boldsymbol{\omega})^{-1}=\lambda_{\boldsymbol{\omega}}\exp_{\Lambda_\psi}\big(\lambda_{\boldsymbol{\omega}}(\bu\cdot\boldsymbol{\omega})\big)^{-1}\in\oK\lambda_{\boldsymbol{\omega}}.
        \]
        Here, the first equality follows from \cite[Cor. 13.7]{BBP21} and the third equality follows from \cite[\S4]{Gek89}. Since $\lambda_{\boldsymbol{\omega}}$ is independent of the choice of elements in $\mathcal{T}_N$, the second assertion follows immediately.
    \end{proof}

    The following proposition is essential to prove Theorem~\ref{Thm:Special_Values_to_Periods}.

    \begin{proposition}\label{P:coefficients} 
        Let $f\in \mathcal{A}\mathcal{M}_k^{r}(\Gamma(N))$. Then for each $0\leq i \leq n-1$, there exists $\mathfrak{q}_i(\underline{X_{\mathbf{u}}})\in \overline{K}[X_{\mathbf{u}} \mid \mathbf{u}\in \mathcal{T}_{\theta}]$ so that 
        \[
            \mathfrak{q}_0(\underline{E^{\mathrm{ari}}_{\mathbf{u}}})+\mathfrak{q}_1(\underline{E^{\mathrm{ari}}_{\mathbf{u}}})f+\dots+\mathfrak{q}_{n-1}(\underline{E^{\text{ari}}_{\mathbf{u}}})f^{n-1}+f^n=0.
        \] 
        In particular, $f$ is integral over $\oK[E_\bu^{\mathrm{ari}}\mid\bu\in\mathcal{T}_\theta]$.
    \end{proposition}
    
    \begin{proof}
    Let $f\in \mathcal{A}\mathcal{M}_k^{r}(\Gamma(N))$. By Lemma \ref{L:integral}, we know that $f$ is integral over the $\mathbb{C}_\infty$-algebra $\mathbb{C}_{\infty}[E_{\mathbf{u}}^{\mathrm{ari}} \mid \mathbf{u}\in \mathcal{T}_{\theta}]$. Hence, there is a positive integer $n$ and, for each $0\leq i \leq n-1$, a polynomial $\mathfrak{p}_i(\underline{X_{\mathbf{u}}})\in \mathbb{C}_{\infty}[X_{\mathbf{u}} \mid \mathbf{u}\in \mathcal{T}_{\theta}]$ such that 
    \begin{equation}\label{E:integraleq}
        \mathfrak{p}_0(\underline{E^{\mathrm{ari}}_{\mathbf{u}}})+\mathfrak{p}_1(\underline{E^{\mathrm{ari}}_{\mathbf{u}}})f+\dots+\mathfrak{p}_{n-1}(\underline{E^{\text{ari}}_{\mathbf{u}}})f^{n-1}+f^n=0.
    \end{equation}
    Here, we mean by $\mathfrak{p}_i(\underline{E^{\mathrm{ari}}_{\mathbf{u}}})$ the evaluation of the polynomial $\mathfrak{p}_i(\underline{X_{\mathbf{u}}})$ at the Eisenstein series $E^{\mathrm{ari}}_{\mathbf{u}}$ for each $\mathbf{u}\in \mathcal{T}_{\theta}$. More precisely, if we set 
    \[
        \mathfrak{p}_i(\underline{X_{\mathbf{u}}}):=\sum_{\substack{\epsilon=(\epsilon_{\mathbf{u}})_{\bu\in \mathcal{T}_{\theta}}}}c_{\epsilon}\prod_{\mathbf{u}\in \mathcal{T}_{\theta}}X_{\mathbf{u}}^{\epsilon_{\mathbf{u}}},~c_{\epsilon}\in\oK
    \]
    then we have 
    \[
        \mathfrak{p}_i(\underline{E^{\mathrm{ari}}_{\mathbf{u}}}):=\sum_{\substack{\epsilon=(\epsilon_{\mathbf{u}})_{\bu\in \mathcal{T}_{\theta}}}}c_{\epsilon}\prod_{\mathbf{u}\in \mathcal{T}_{\theta}}(E^{\mathrm{ari}}_{\mathbf{u}})^{\epsilon_{\mathbf{u}}},~c_{\epsilon}\in\oK.
    \]
    Now assume without loss of generality, that is, by passing to a higher level if necessary, that $\theta|N$ and hence by Lemma \ref{Lem:Level_Change},  $E_{\bu}^{\mathrm{ari}}\in \mathcal{M}^r({\Gamma(N)})$. Since $\mathcal{M}^r({\Gamma(N)})$ is a graded ring by weights, by taking the pure weight $kn$ parts of \eqref{E:integraleq}, we have $\mathfrak{p}_i(\underline{E^{\mathrm{ari}}_{\mathbf{u}}})\in \mathcal{M}^r_{k(n-i)}(\Gamma(N))$ and $\mathfrak{p}_i(\underline{X_{\mathbf{u}}})$ is a homogeneous polynomial in $X_{\mathbf{u}}$ of degree $k(n-i)$. 
    
    For $0\leq i\leq n-1$, we consider
    \[
        \mathfrak{P}_i(\underline{X_{\mathbf{u}}},\underline{Y_{\epsilon}}):=\sum_{\substack{\epsilon=(\epsilon_{\mathbf{u}})_{\mathbf{u}\in \mathcal{T}_{\theta}}}, \wt(\epsilon)=ik}Y_{\epsilon}\prod_{\mathbf{u}\in \mathcal{T}_{\theta}}X_{\mathbf{u}}^{\epsilon_{\mathbf{u}}}\in \mathbb{F}_q[X_{\mathbf{u}}, Y_{\epsilon} \mid  \mathbf{u}\in\mathcal{T}_{\theta},~\epsilon\in\mathbb{Z}_{\geq 0}^{|\mathcal{T}_\theta|},~\wt(\epsilon)=k(n-i)],
    \]
    where for each $\epsilon=(\epsilon_{\bu})_{\bu\in\mathcal{T}_\theta}$, we set $\wt(\epsilon):=\sum_{\bu\in\mathcal{T}_\theta}\epsilon_\bu$.
    Note that
    \[
        \mathfrak{P}_i(\underline{X_{\mathbf{u}}},\underline{Y_{\epsilon}})|_{X_{\mathbf{u}}=E^{\mathrm{ari}}_{\mathbf{u}}}=\mathfrak{P}_i(\underline{E_{\mathbf{u}}^{\mathrm{ari}}},\underline{Y_{\epsilon}})\in (\mathbb{F}_q[E_\bu^{\mathrm{ari}}\mid\bu\in\mathcal{T}_\theta])[Y_{\epsilon}\mid\epsilon\in\mathbb{Z}_{\geq 0}^{|\mathcal{T}_\theta|},~\wt(\epsilon)=k(n-i)].
    \]
    Consider
    \begin{equation}\label{E:SLE}
        \mathfrak{P}_0(\underline{E^{\mathrm{ari}}_{\mathbf{u}}},\underline{Y_{\epsilon}})+\mathfrak{P}_1(\underline{E^{\mathrm{ari}}_{\mathbf{u}}},\underline{Y_{\epsilon}})f+\dots+\mathfrak{P}_{n-1}(\underline{E^{\mathrm{ari}}_{\mathbf{u}}},\underline{Y_{\epsilon}})f^{n-1}+f^n=0.
    \end{equation}

    In what follows, we will use the injection described in \eqref{E:injection} to construct a system of linear equations over $\overline{K}$ in the variables $\underline{Y_{\epsilon}}$. Note that, by \eqref{E:injection}, since $f$ and $E_{\mathbf{u}}^{\mathrm{ari}}$ are elements in $\mathcal{AM}^r(\Gamma(N))$, for each $\ell\geq 0$ and $\mathbf{u}\in \mathcal{T}_{\theta}$, there exist uniquely defined rigid analytic functions $f_{\ell},\widetilde{F}_{\mathrm{u},\ell}:\Omega^{r-1}\to \mathbb{C}_{\infty}$ such that 
\begin{equation}\label{E:expan1}
\mathcal{F}_{f}=\sum_{\ell\geq 0}f_{\ell}X^{\ell}\in \mathcal{O}(\Omega^{r-1})\llbracket X\rrbracket
\end{equation}
and
\begin{equation}\label{E:expan2}
\mathcal{F}_{E^{\mathrm{ari}}_{\mathbf{u}}}=\sum_{\ell\geq 0}\widetilde{F}_{\mathrm{u},\ell}X^{\ell}\in \mathcal{O}(\Omega^{r-1})\llbracket X\rrbracket.
\end{equation}

For each $0\leq \mu \leq n-1$, 
using \eqref{E:expan1} and \eqref{E:expan2}, we can see that each coefficient of the monomials in $\mathfrak{P}_{\mu}(\underline{E^{\mathrm{ari}}_{\mathbf{u}}},Y_{\underline{\epsilon}})f^{\mu}\in \mathcal{O}(\Omega^{r-1})[Y_{\epsilon} \mid  \epsilon\in\mathbb{Z}_{\geq 0}^{|\mathcal{T}_\theta|}\mid\wt(\epsilon)=k(n-\mu)]$ leads to an infinite series expansion in $X$. Hence, by a slight abuse of notation, we can write  
\[
\mathcal{F}_{\mathfrak{P}_{\mu}(\underline{E^{\mathrm{ari}}_{\mathbf{u}}},Y_{\underline{\epsilon}})f^{\mu}}=\sum_{\ell\geq 0}F_{r,\mu,\ell}X^{\ell}\in \Span_{\mathcal{O}(\Omega^{r-1})}(1,Y_{\underline{\epsilon}})\llbracket X\rrbracket
\]
for some rigid analytic functions $F_{r,\mu,\ell}:\Omega^{r-1}\to \mathbb{C}_{\infty}$. Now comparing the $X^{\ell}$ coefficients of both sides of \eqref{E:SLE}, one obtains a linear equation in the variables $Y_{\epsilon}$ given by
\begin{equation}\label{E:LE1}
F_{r,0,\ell}+\cdots+F_{r,n,\ell}=0.
\end{equation}
Thus, since $f$ is an arithmetic Drinfeld modular form, using the induction argument in the proof of \cite[Lem. 2.6.6]{CG25}, for each linear equation in \eqref{E:LE1}, one can obtain a linear equation, hence a linear system $\mathcal{S}$, over $\overline{K}$ in $Y_{\underline{\epsilon}}$. Since, the set
\[
    \cup_{i=0}^{n-1}\{c_{\epsilon} \mid \epsilon\in \mathbb{Z}_{\geq 0}^{|\mathcal{T}_\theta|},~\wt(\epsilon)=k(n-i) \}
\]
is a solution for $\mathcal{S}$, we know that $\mathcal{S}$ is solvable system of linear equations over $\oK$ and it has at most 
\[
    |\cup_{i=0}^{n-1}\{c_{\epsilon} \mid \epsilon\in \mathbb{Z}_{\geq 0}^{|\mathcal{T}_\theta|},~\wt(\epsilon)=k(n-i) \}|<\infty
\]
many linearly independent equations. Since $\mathcal{S}$ is defined over $\overline{K}$, there exists a set of solutions
\[
    \cup_{i=0}^{n-1}\{d_{\epsilon} \mid \epsilon\in \mathbb{Z}_{\geq 0}^{|\mathcal{T}_\theta|},~\wt(\epsilon)=k(n-i) \}\subset\oK.
\]
Finally, we define
\[
    \mathfrak{q}_i(\underline{X_{\mathbf{u}}}):=\mathfrak{P}_i(\underline{X_{\mathbf{u}}},\underline{d_{\epsilon}})\in\oK[X_\bu\mid\bu\in\mathcal{T}_\theta].
\]
Then it follows from our construction and \eqref{E:SLE} that
\begin{align*}
    0&=\big(\mathfrak{P}_0(\underline{E^{\mathrm{ari}}_{\mathbf{u}}},\underline{Y_{\epsilon}})+\mathfrak{P}_1(\underline{E^{\mathrm{ari}}_{\mathbf{u}}},\underline{Y_{\epsilon}})f+\dots+\mathfrak{P}_{n-1}(\underline{E^{\mathrm{ari}}_{\mathbf{u}}},\underline{Y_{\epsilon}})f^{n-1}+f^n\big)\mid_{Y_\epsilon=d_\epsilon}\\
    &=\mathfrak{q}_0(\underline{E^{\mathrm{ari}}_{\mathbf{u}}})+\mathfrak{q}_1(\underline{E^{\mathrm{ari}}_{\mathbf{u}}})f+\dots+\mathfrak{q}_{n-1}(\underline{E^{\text{ari}}_{\mathbf{u}}})f^{n-1}+f^n
\end{align*}
which finishes the proof of the proposition. 
\end{proof}

    Let $\alpha,\beta\in\mathbb{C}_\infty$. We denote by $\alpha\sim\beta$ if the ratio $\alpha/\beta\in\oK$. In order to state the main theorem of this section, we need the following definition.


    Now we are ready to present the primary goal of this section.
    
    \begin{theorem}\label{Thm:Special_Values_to_Periods}
        Let $F$ be a meromorphic arithmetic Drinfeld modular form of non-zero weight $k$ for a congruence subgroup $\Gamma$ of $\GL_r(A)$. For any CM point $\boldsymbol{\omega}\in\Omega^r$, if $F$ is defined at $\boldsymbol{\omega}$, then we have
        \[
            F(\boldsymbol{\omega})\sim\left(\lambda_{\boldsymbol{\omega}}/\widetilde{\pi}\right)^k
        \]
        where $\lambda_{\boldsymbol{\omega}}$ is a period of a CM Drinfeld module  defined over $\oK$.
    \end{theorem}
    
    \begin{proof} We first choose a monic polynomial $N\in A$ such that it is divisible by $\theta$ and $\Gamma(N)\subset \Gamma$. Let $F\in \mathcal{M}_k^r(\Gamma(N))$. By Lemma \ref{L:integral} and Proposition \ref{P:coefficients}, we know that there exist a positive integer $n$ and, for each $0\leq i \leq n-1$, a homogeneous polynomial $\mathfrak{q}_i(\underline{X_{\mathbf{u}}})\in \overline{K}[X_{\mathbf{u}} \ \ | \mathbf{u}\in \mathcal{T}_{\theta}]$ of degree $k(n-i)$ such that 
    \begin{equation}\label{E:integ2}
        \mathfrak{q}_0(\underline{E^{\mathrm{ari}}_{\mathbf{u}}})+\mathfrak{q}_1(\underline{E^{\mathrm{ari}}_{\mathbf{u}}})F+\dots+\mathfrak{q}_{n-1}(\underline{E^{\text{ari}}_{\mathbf{u}}})F^{n-1}+F^n=0.
    \end{equation}
    Since each $\mathfrak{q}_i(X_{\mathbf{u}})$ is homogeneous of degree $k(n-i)$, we have 
    \[
        \frac{\mathfrak{q}_i(\underline{E^{\mathrm{ari}}_{\mathbf{u}}})}{(E_\mathbf{u}^{\mathrm{ari}})^{k(n-i)}}\in \overline{K}\left[\frac{E_{\mathbf{u}_1}}{E_{\mathbf{u}_2}} \ \ | \ \ \mathbf{u}_1,\mathbf{u}_2\in \mathcal{T}_{\theta}\right]
    \] 
    and hence by Lemma \ref{Lem:CM_values_E}, we obtain 
    \begin{equation}\label{E:integ3}
        \frac{\mathfrak{q}_i(\underline{E^{\mathrm{ari}}_{\mathbf{u}}(\bomega))}}{(E_\mathbf{u}^{\mathrm{ari}}(\bomega))^{k(n-i)}}\in \overline{K}.
    \end{equation}

    We now evaluate the left hand side of \eqref{E:integ2} at $\bomega$. Then dividing it by $(E_{\mathbf{u}}^{\mathrm{ari}}(\bomega))^{nk}$ and using \eqref{E:integ3}, we see that 
    \[
        F(\bomega)/E_{\mathbf{u}}^{\mathrm{ari}}(\bomega)^k\in \overline{K}.
    \]
    In other words, $F(\boldsymbol{\omega})\sim E_{\bu}^{\mathrm{ari}}(\boldsymbol{\omega})^k$ and hence using Lemma~\ref{Lem:CM_values_E} once again, we obtain
    \begin{equation}\label{E:algebraicity}
        F(\boldsymbol{\omega})\sim\left(\lambda_{\boldsymbol{\omega}}/\widetilde{\pi}\right)^k,
    \end{equation}
    as desired. On the other hand, if we write $F=\frac{F_1}{F_2}$ for some elements $F_1\in \mathcal{A}\mathcal{M}^r_{k+\ell}(\Gamma(N))$ and $F_2\in \mathcal{A}\mathcal{M}^r_{\ell}(\Gamma(N))$ with $\ell\geq 1$, we see, by Lemma \ref{Lem:CM_values_E} and \eqref{E:algebraicity}, that 
    \[
        F(\bomega)=\frac{F_1(\bomega)}{F_2(\bomega)}=(E_{\bu}^{\mathrm{ari}}(\bomega))^k\frac{\frac{F_1(\bomega)}{(E_{\bu}^{\mathrm{ari}}(\bomega))^{k+\ell}}}{\frac{F_2(\bomega)}{(E_{\bu}^{\mathrm{ari}}(\bomega))^{\ell}}}\sim \left(\lambda_{\boldsymbol{\omega}}/\widetilde{\pi}\right)^{k},
    \]
    which concludes the proof of the theorem.
\end{proof}

\section{$t$-motivic Galois groups and algebraic independence}
    In what follows, our main purpose is to determine the transcendence degree of the field generated over $\oK$ by periods and quasi-periods of finitely many CM Drinfeld modules defined over $\oK$ under some natural assumption. We will begin with some essential properties of linear algebraic groups.

\subsection{Properties of algebraic tori}
    Let $\mathcal{K}_1$ and $\mathcal{K}_2$ be two Galois extensions over $\mathbb{F}_q(t)$ of degree $r_1\geq 2$ and $r_2\geq 2$. Assume that $\mathcal{K}_1$ and $\mathcal{K}_2$ are linearly disjoint over $\mathbb{F}_q(t)$. For each $i=1,2$,  let $G_i:=\Res_{\mathcal{K}_i/\mathbb{F}_q(t)}(\mathbb{G}_{m/\mathcal{K}_i})$ be the Weil restriction of scalars of the multiplicative group and $H_i:=\Res_{\mathcal{K}_i/\mathbb{F}_q(t)}(\mathbb{G}_{m/\mathcal{K}_i})^{(1)}$ be the norm torus inside $G_i$. That is, we have the following short exact sequence of linear algebraic groups over $\mathbb{F}_q(t)$:
        \[
            1\to H_i\hookrightarrow G_i\twoheadrightarrow \mathbb{G}_m\to 1.
        \]
        Here the map between $H_i$ and $G_i$ is the inclusion map and the map between $G_i$ and $\mathbb{G}_m$ is the determinant map which coincides with the field norm $N_{\mathcal{K}_i/\mathbb{F}_q(t)}$.
        
    To achieve our goal, we need the following extension of \cite[Prop.~3.3.1]{Cha12}.
    \begin{lemma}\label{Lem:Morphism_between_Torus}
         Any $\mathbb{F}_q(t)$-morphism (as morphisms of $\mathbb{F}_q(t)$-varieties) $f:H_1\to H_2$ is a constant map.
    \end{lemma}
    
    \begin{proof}
        For each $\mathbb{F}_q(t)$-algebraic group $G$, we set $(G)_L$ to be the extension of scalars of $G$ to a finite extension $L/\mathbb{F}_q(t)$. Since $\mathcal{K}_1$ is Galois over $\mathbb{F}_q(t)$, $(G_1)_{\mathcal{K}_1}$ and $(H_1)_{\mathcal{K}_1}$ both are split. We claim that $(H_2)_{\mathcal{K}_1}$ is anisotropic. Since $\mathcal{K}_2/\mathbb{F}_q(t)$ is finite and separable, we may write $\mathcal{K}_2=\mathbb{F}_q(t)(\beta)$ for some $\beta\in\mathcal{K}_2$. If we set $f_\beta(X)\in \mathbb{F}_q(t)[X]$ to be the minimal polynomial of $\beta$ over $\mathbb{F}_q(t)$, then
        \[
            \mathcal{K}_2\cong \mathbb{F}_q(t)[X]/(f_\beta(X))
        \]
        and
        \begin{align*}
            \mathcal{K}_2\otimes_{\mathbb{F}_q(t)}\mathcal{K}_1\cong \mathcal{K}_1[X]/(f_\beta(X))\cong \mathcal{K}_1[X]/(f_{\beta,1}(X))\times\cdots\times \mathcal{K}_1[X]/(f_{\beta,m}(X))
        \end{align*}
        where $f_\beta(X)=\prod_{i=1}^mf_{\beta,i}(X)\in\mathcal{K}_1[X]$ and $f_{\beta,i}(X)$ is irreducible in $\mathcal{K}_1[X]$ for $1\leq i \leq m$. Since $\mathcal{K}_1$ and $\mathcal{K}_2$ are linearly disjoint over $\mathbb{F}_q(t)$, $\mathcal{K}_2\otimes_{\mathbb{F}_q(t)}\mathcal{K}_1$ is isomorphic to the composite of fields $\mathcal{K}_1\mathcal{K}_2$. It implies that $m=1$ and $f_{\beta}(X)$ remains irreducible in $\mathcal{K}_1[X]$. It follows that $(G_2)_{\mathcal{K}_1}\cong\Res_{\mathcal{K}_1\mathcal{K}_2}(\mathbb{G}_{m/\mathcal{K}_1\mathcal{K}_2})$ and $(H_2)_{\mathcal{K}_1}\cong\Res_{\mathcal{K}_1\mathcal{K}_2/\mathcal{K}_1}(\mathbb{G}_{m/\mathcal{K}_1\mathcal{K}_2})^{(1)}$ is anisotropic as desired.
        
        
        Now, we consider any $k$-morphism $f:H_1\to H_2$. Then the extension of scalars induces the map
        \[
            f\otimes_k\mathcal{K}_1:(H_1)_{\mathcal{K}_1}\cong\prod\mathbb{G}_{m/\mathcal{K}_1}\to (H_2)_{\mathcal{K}_1}\cong\Res_{\mathcal{K}_1\mathcal{K}_2/\mathcal{K}_1}(\mathbb{G}_{m/\mathcal{K}_1\mathcal{K}_2})^{(1)}.
        \]
        By \cite[Prop.~13.2.2]{Spr98}, we conclude that $f$ must be a constant map.
    \end{proof}
    
    In what follows, for each $1\leq i \leq n$, let $\mathcal{K}_i$ be a Galois extension over $\mathbb{F}_q(t)$ of degree $r_i\geq 2$. Moreover, let $G_i:=\Res_{\mathcal{K}_i/\mathbb{F}_q(t)}(\mathbb{G}_{m/\mathcal{K}_i})$ and $H_i:=\Res_{\mathcal{K}_i/\mathbb{F}_q(t)}(\mathbb{G}_{m/\mathcal{K}_i})^{(1)}$. We set
    \[
        T_n:=\{(\gamma_1,\dots,\gamma_n)\in G_1\times\dots \times G_n\mid \det(\gamma_i)=\det(\gamma_j)\}.
    \]
    Then $T_n$ defines an $\mathbb{F}_q(t)$-torus of dimension $\dim(T_n)=(r_1+\cdots+r_n)-(n-1)$.
    Let 
    \[
        \pi_{[n]}:G_1\times\cdots\times G_n\twoheadrightarrow G_1\times\cdots\times G_{n-1}
    \]
    be the natural projection onto the first $n-1$ copies and let
    \[
        \pi_{n}:G_1\times\cdots\times G_n\twoheadrightarrow G_n
    \]
    be the natural projection onto $G_n$. Note that $\pi_{[n]}:T_n\to T_{n-1}$ is surjective. To show this, for a given element $(\gamma_1,\dots,\gamma_{n-1})\in T_{n-1}$, we first let $D:=\det(\gamma_1)=\cdots=\det(\gamma_{n-1})\in\mathbb{G}_m$. Then, since $\det:G_n\to\mathbb{G}_m$ is surjective, there is $\gamma_n\in G_n$ such that $\det(\gamma_n)=D$. Therefore, $(\gamma_1,\dots,\gamma_n)\in T_n$ and $\pi_{[n]}(\gamma_1,\dots,\gamma_n)=(\gamma_1,\dots,\gamma_{n-1})$.
    
    The next theorem generalizes \cite[Thm.~3.3.2]{Cha12}.
    
    \begin{theorem}\label{Thm:Galois}
        Let $n\geq 2$ and assume that $\mathcal{K}_1,\dots,\mathcal{K}_n$ are pairwise linearly disjoint over $\mathbb{F}_q(t)$.
        Suppose that $\mathcal{G}\subset T_n$  is an $\mathbb{F}_q(t)$-subtorus such that $\pi_{[n]}(\mathcal{G})=T_{n-1}$ and $\pi_n(\mathcal{G})=G_n$. Then we have $\mathcal{G}=T_n$.
    \end{theorem}
    
    \begin{proof}
        
        Let $\mathcal{G}_{[n]}:=\Ker\pi_{[n]}\cap\mathcal{G}$. Then $\mathcal{G}_{[n]}$ is of the form
        \begin{equation}\label{Eq:Gamma_n}
            \mathcal{G}_{[n]}=\left\{\left(
            \begin{array}{ccc|c}
                \mathbb{I}_{r_1} &  &  &\\
                & \ddots & & \\ 
                & & \mathbb{I}_{r_n} &\\ \hline
                & & & (\star)
            \end{array}\right)\mid
            (\star)\in H_n\right\}\subset\mathcal{G}
        \end{equation}
        and $\mathcal{G}_{[n]}$ forms an $\mathbb{F}_q(t)$-subtorus of $\mathcal{G}$. 
        Consequently, we have the following short exact sequence of linear algebraic groups over $\mathbb{F}_q(t)$:
        \[
            1\to \mathcal{G}_{[n]}\hookrightarrow \mathcal{G}\overset{\pi_{[n]}}{\twoheadrightarrow} T_{n-1}\to 1.
        \]
        By \cite[Prop.~13.2.3]{Spr98}, there exists an $\mathbb{F}_q(t)$-subtorus $\mathcal{G}_{[n]}'\subset\mathcal{G}$ such that $\mathcal{G}=\mathcal{G}_{[n]}\cdot\mathcal{G}_{[n]}'$ and $|\mathcal{G}_{[n]}\cap\mathcal{G}_{[n]}'|<\infty$. In particular, we have
        \[
            \dim\mathcal{G}=\dim\mathcal{G}_{[n]}+\dim\mathcal{G}_{[n]}'=\dim\mathcal{G}_{[n]}+\dim T_{n-1}.
        \]
        Thus, $\dim\mathcal{G}_{[n]}'=\dim T_{n-1}=r_1+\cdots+r_{n-1}-(n-2)$ and in fact we have $\pi_{[n]}(\mathcal{G}_{[n]}')=T_{n-1}$. Indeed, given $\mathbf{y}\in T_{n-1}$, since $\pi_{[n]}(\mathcal{G})=T_{n-1}$, there is an element $\mathbf{x}\in\mathcal{G}$ so that $\pi_{[n]}(\mathbf{x})=\mathbf{y}$. Note that $\mathcal{G}=\mathcal{G}_{[n]}\cdot\mathcal{G}_{[n]}'$. It follows that there exist $\mathbf{x}_1\in\mathcal{G}_{[n]}$ and $\mathbf{x}_2\in\mathcal{G}_{[n]}'$ satisfying $\mathbf{x}=\mathbf{x}_1\mathbf{x}_2$. By \eqref{Eq:Gamma_n}, the top-left blocks of $\mathbf{x}_1$ are given by the identity matrix. Since $\pi_{[n]}$ is the projection onto the first $n-1$ copies, we must have $\pi_{[n]}(\mathbf{x}_2)=\mathbf{y}$. This gives the desired surjectivity.
        
        Now, by \cite[Prop.~1.3.1]{Ono61}, there exists an $\mathbb{F}_q(t)$-isogeny $\pi_{[n]}^{\vee}:T_{n-1}\twoheadrightarrow\mathcal{G}_{[n]}'$ such that $\pi_{[n]}^{\vee}\circ(\pi_{[n]}\mid_{\mathcal{G}_{[n]}'})={\deg(\pi_{[n]}\mid_{\mathcal{G}_{[n]}'})}$. In particular, we have a homomorphism of linear algebraic groups $\pi_{n}\circ\pi_{[n]}^{\vee}$ which is defined over $\mathbb{F}_q(t)$. We shall mention that this map is not necessary to be surjective. We claim that $(\pi_{n}\circ\pi_{[n]}^{\vee})(H_1\times\cdots\times H_{n-1})\subset H_n$. To see this, let $\mathbf{x}=(x_1,\dots,x_{n-1})\in H_1\times\cdots\times H_{n-1}$. Since $\pi_{[n]}\mid_{\mathcal{G}_{[n]}'}$ is surjective, there exists $\mathbf{y}=(x_1,\dots,x_{n-1},y)\in\mathcal{G}_{[n]}'$ such that $\pi_{[n]}(\mathbf{y})=\mathbf{x}$. Then
        \begin{align*}
            (\pi_n\circ\pi_{[n]}^{\vee})(\mathbf{x})=\pi_n\circ(\pi_{[n]}^{\vee}\circ\pi_{[n]}\mid_{\mathcal{G}_{[n]}'})(\mathbf{y})=\pi_n(\deg(\pi_{[n]}\mid_{\mathcal{G}_{[n]}'})(\mathbf{y}))\in H_n.
        \end{align*}
        Thus, we have $(\pi_n\circ\pi_{[n]}^{\vee}):H_1\times\cdots\times H_{n-1}\to H_n$ and it must be the trivial map by Lemma~\ref{Lem:Morphism_between_Torus}.
        
        Now, we consider the short exact sequence of linear algebraic groups
        $$1\to \Ker\pi_n\hookrightarrow \mathcal{G}\overset{\pi_{n}}{\twoheadrightarrow} G_n\to 1.$$
        On the one hand, we have
        \begin{align*}
            \dim\mathcal{G}&=\dim\Ker\pi_n+\dim G_n\\
            &\geq \dim\pi_{[n]}^{\vee}(H_1\times\cdots\times H_{n-1})+r_n\\
            &=\dim (H_1\times\cdots\times H_{n-1})+r_n\\
            &=(r_1-1)+\cdots+(r_{n-1}-1)+r_n\\
            &=(r_1+\cdots+r_n)-(n-1).
        \end{align*}
        Here the first equality follows from \cite[A.99]{Mil15}, the second inequality follows from the fact that $(\pi_n\circ\pi_{[n]}^{\vee}):H_1\times\cdots\times H_{n-1}\to H_n$ is the trivial map and the third equality follows from the fact that $\pi_{[n]}^{\vee}$ is an isogeny.
        On the other hand, $\mathcal{G}\subset T_n$ implies that
        \[
            \dim\mathcal{G}\leq\dim T_n=(r_1+\cdots+r_n)-(n-1).
        \]
        Thus, we have $\dim\mathcal{G}=\dim T_n=(r_1+\cdots+r_n)-(n-1)$. Finally, the connectedness of $\mathcal{G}$ and $T_n$ implies that $\mathcal{G}=T_n$ as desired.
    \end{proof}

\subsection{Tannakian formalism for rigid analytically trivial pre-$t$-motives}
    In what follows, we give a rapid review on the $t$-motivic Galois theory developed by Papanikolas in \cite{Pap08}. Consider the twisted Laurent polynomial ring $\oK(t)[\sigma,\sigma^{-1}]$ in $\sigma$ over $\oK(t)$, subject to the relation $\sigma^nf=f^{(-n)}\sigma^n$ for any $n\in\mathbb{Z}$. A left $\oK(t)[\sigma,\sigma^{-1}]$-module $M$ is called a \emph{pre-$t$-motive} if it is finite dimensional over $\oK(t)$. For two pre-$t$-motives $M_1$ and $M_2$, we call  $f:M_1\to M_2$ \emph{a morphism of pre-$t$-motives} if $f$ is a left $\oK(t)[\sigma,\sigma^{-1}]$-module homomorphism.

    Since a pre-$t$-motive $M$ is a finite dimensional $\oK(t)$-vector space, we may set $r:=\dim_{\oK(t)}M$ and fix a choice of a $\oK(t)$-basis $\bm\in\Mat_{r\times 1}(M)$. Then there is a matrix $\Phi\in\GL_r(\oK(t))$ representing the action of $\sigma$ in the sense that $\sigma\bm=\Phi\bm$. We denote by $\LL$ the field of fractions of $\TT$. The pre-$t$-motive $M$ is called \emph{rigid analytically trivial} if there exists $\Psi\in\GL_r(\LL)$ such that $\Psi^{(-1)}=\Phi\Psi$. We may extend the $\sigma$-action on $\LL\otimes_{\oK(t)}M$ by setting $\sigma(f\otimes m):=f^{(-1)}\otimes \sigma m$, and we further define $M^B:=(\LL\otimes_{\oK(t)}M)^{\sigma}$ the $\mathbb{F}_q(t)$-subspace of $\sigma$-invariant elements in $\LL\otimes_{\oK(t)}M$. Denoting by $\mathcal{R}$ the category of rigid analytically trivial pre-$t$-motives,  by \cite[Thm.~3.3.15]{Pap08}, we see that $\mathcal{R}$ forms a neutral Tannakian category over $\mathbb{F}_q(t)$ with fiber functor $M\mapsto M^B$.

   Recall the dual $t$-motives defined in \S2.1 and let $\mathcal{AR}$ be the category of rigid analytically trivial dual $t$-motives. We define $\mathcal{T}$ to be \emph{the category of the $t$-motives} which is given by the essential image of the functor $\mathcal{M}\mapsto\oK(t)\otimes_{\oK[t]}\mathcal{M}$ from $\mathcal{AR}$ into $\mathcal{R}$.
    For a given $t$-motive $M$, we may consider $\mathcal{T}_M$, the strictly full Tannakian subcategory of $\mathcal{T}$ generated by $M$. Then, by the Tannakian duality, $\mathcal{T}_M$ is equivalent to $\mathbb{F}_q(t)$-finite dimensional representations of an affine algebraic group scheme $\Gamma_M$ defined over $\mathbb{F}_q(t)$. We call $\Gamma_M$ \emph{the $t$-motivic Galois group of $M$}. If $\mathcal{M}$ is a rigid analytically trivial dual $t$-motive and $(\Phi,\Psi)$ is a rigid analytic trivialization of $\mathcal{M}$, then, letting $M:=\oK(t)\otimes_{\oK[t]}\mathcal{M}$, we have
    \begin{equation}\label{Eq:trdeg_periods}
        \dim\Gamma_M=\trdeg_{\oK}\oK(\Psi(\theta))
    \end{equation}
    which can be regarded as a function field analogue of Grothendieck’s periods conjecture established by Papanikolas \cite[Thm.~1.1.7]{Pap08}. Note that the $t$-motivic Galois group $\Gamma_M$ has the following explicit description: We define $\widetilde{\Psi}:=\Psi_1^{-1}\Psi_2\in\GL_r(\LL\otimes_{\oK(t)}\LL)$ where $(\Psi_1)_{ij}=\Psi_{ij}\otimes 1$ and $(\Psi_2)_{ij}=1\otimes\Psi_{ij}$. Then as an affine algebraic group scheme over $\mathbb{F}_q(t)$, we have an isomorphism
    \[
        \Gamma_M\cong\Spec\mathrm{Im}(\mu_\Psi)
    \]
    where $\mu_\Psi:\mathbb{F}_q(t)[X,1\det(X)]\to\LL\otimes_{\oK(t)}\LL$ is an $\mathbb{F}_q(t)$-algebra homomorphism defined by $X_{ij}\mapsto\widetilde{\Psi}_{ij}$ with $X=(X_{ij})$ a square matrix of size $r$ in variables $X_{ij}$.

    Now, for each $1\leq i\leq n$, let $\phi_i$ be a CM Drinfeld module defined over $\oK$. Its dual $t$-motive $\mathcal{M}_i$ has a rigid analytic trivialization $(\Phi_i,\Psi_i)$. Then the direct sum   $\mathcal{M}:=\oplus_{1\leq i\leq n}\mathcal{M}_i$  has a rigid analytic trivialization $(\Phi,\Psi)$ where
    \[
        \Phi:=\begin{pmatrix}
            \Phi_1 & & & \\
             & \Phi_2 & & \\
             & & \ddots & \\
             & & & \Phi_n
        \end{pmatrix}\in\Mat_{r_1+\cdots+r_n}(\oK[t])\cap\GL_{r_1+\cdots+r_n}(\oK(t)),
    \]
    and
    \[
        \Psi:=\begin{pmatrix}
            \Psi_1 & & & \\
             & \Psi_2 & & \\
             & & \ddots & \\
             & & & \Psi_n
        \end{pmatrix}\in\GL_{r_1+\cdots+r_n}(\TT_\theta).
    \]
    Since, by \eqref{E:det}, $\det\widetilde{\Psi_i}=\Omega(t)^{-1}\otimes\Omega(t)\in\LL\otimes_{\oK(t)}\LL$, letting $M:=\oK(t)\otimes_{\oK[t]}\mathcal{M}$ and $M_i:=\oK(t)\otimes_{\oK[t]}\mathcal{M}_i$, we must have (cf. \cite[p.~202]{Cha12})
    \begin{equation}\label{Eq:Galois_1}
        \Gamma_M\subset T_n\subset \Gamma_{M_1}\times \cdots \times \Gamma_{M_n}.
    \end{equation}
     In addition, by the Tannakian theory (cf. \cite[Rem.~3.1.3]{Cha12}), the canonical projection $\pi_n:\Gamma_{M}\to\Gamma_{M_n}$ is surjective. Moreover, if we set $M_{[n]}:=\oplus_{i=1}^{n-1}M_i$ so that $M= M_{[n]}\oplus M_{n} $, then by the Tannakian theory again there is a natural surjective map from $\Gamma_M\to\Gamma_{M_{[n]}}$ which coincides with the canonical projection map $\pi_{[n]}$. Finally, if we denote by $\mathcal{K}_i:=\End(\phi_i)\otimes_{\mathbb{F}_q[t]}\mathbb{F}_q(t)$ the endomorphism algebra of $\phi_i$, then by \cite[Thm.3.5.4]{CP12}, we have
    \begin{equation}\label{Eq:Galois_2}
        \Gamma_{M_i}\cong\Res_{\mathcal{K}_i/\mathbb{F}_q(t)}(\mathbb{G}_{m/\mathcal{K}_i}).
    \end{equation}

    Now we are ready to state and prove the main result of this subsection, which generalizes \cite[Thm.~2.2.2]{Cha12}.
    \begin{theorem}\label{Thm:Algebraic_Independence}
        Let $\phi_1,\dots,\phi_n$ be CM Drinfeld modules of rank $r_1,\dots,r_n$ defined over $\oK$. Suppose that for any $1\leq i \leq n$, the endomorphism algebra $\mathcal{K}_i:=\mathbb{F}_q(t)\otimes_{\mathbb{F}_q[t]}\End(\Lambda_{\phi_i})$ is Galois over $\mathbb{F}_q(t)$ and $\mathcal{K}_i\cap\mathcal{K}_j=\mathbb{F}_q(t)$ for any $i\neq j$. Let $\oK\big(\cup_{i=1}^nP_{\phi_i}\big)$ be the field generated by the entries of $P_{\phi_i}$ for each $1\leq i\leq n$ over $\oK$. Then we have
        \[
            \trdeg_{\oK}\oK(\cup_{i=1}^nP_{\phi_i})=(r_1+\cdots+r_n)-(n-1).
        \]
        In particular, for each $1\leq i\leq n$, if we choose $0\neq\lambda_i\in\Lambda_{\phi_i}$, then
        \[
            \trdeg_{\oK}\oK(\lambda_1/\widetilde{\pi},\dots,\lambda_n/\widetilde{\pi})=n.
        \]
    \end{theorem}
   \begin{proof}
        Using \eqref{E:fieldperiod} and the structure of the block diagonal matrix $\Psi$, we have
        \[
            \oK(\cup_{i=1}^nP_{\phi_i})=\oK(\Psi(\theta)).
        \]
        It follows from \cite[Thm.~1.1.7]{Pap08} that
        \[
            \dim\Gamma_M=\trdeg_{\oK}\oK(\Psi(\theta)).
        \]
        Thus, to prove the first assertion of the theorem, it suffices to determine $\dim\Gamma_M$. We first claim that $\Gamma_M=T_n$. We prove our claim by induction on $n$. When $n=1$, we simply have $\Gamma_{M}=\Gamma_{M_1}=T_1$ and hence the claim holds. Assume that the desired assertion holds for $n-1$, that is, we have $\Gamma_{M_{[n]}}=T_{n-1}$ where $M_{[n]}$ is as defined above. Using \eqref{Eq:Galois_1}, we have $\Gamma_M\subset T_n$. On the other hand, since $M_n$ can be identified with a $\oK[t,\sigma]$-submodule of $M$, the Tannakian theory (see also \cite[Rem. 3.1.3]{Cha12}) shows that $\pi_n(\Gamma_M)=\Gamma_{M_n}$. Moreover, as discussed after \eqref{Eq:Galois_1}, we have $\pi_{[n]}(\Gamma_{M})=\Gamma_{M_{[n]}}=T_{n-1}$ where the last equality follows from the induction hypothesis. Then by Theorem~\ref{Thm:Galois}, we have $\Gamma_M=T_n$ as desired. In particular, we have
        \[
            \dim\Gamma_M=\dim T_n=(r_1+\cdots+r_n)-(n-1).
        \]

        To prove the second part, we verify that the set
        \[
            \mathfrak{S}:=\left\{\frac{\lambda_1}{\widetilde{\pi}},\dots,\frac{\lambda_n}{\widetilde{\pi}}\right\}
        \]
        is algebraically independent over $\oK$. By \cite[Prop.~3.1.2]{CG25}, we have
        \[
            \oK(\cup_{i=1}^nP_{\phi_i})=\oK\left(\cup_{i=1}^n\left\{\frac{\lambda_{i}}{\widetilde{\pi}},\frac{F^{\phi_i}_{\tau}(\lambda_{i})}{\widetilde{\pi}},\dots,\frac{F^{\phi_i}_{\tau^{r_i-1}}(\lambda_{i})}{\widetilde{\pi}}\right\}\right).
        \]
        Since $\mathcal{K}_i$ is Galois over $\mathbb{F}_q(t)$ for each $1\leq i\leq n$ and $\mathcal{K}_i\cap\mathcal{K}_j=\mathbb{F}_q(t)$ for any $i\neq j$,  the collection $\{\mathcal{K}_i\}_{i=1}^n$ of fields is pairwise non-isomorphic and hence the collection $\{\phi_i\}_{i=1}^n$ of CM Drinfeld modules is pairwise non-isogenous (\cite[Prop. 2.1.1]{Cha12}). By Lemma \ref{L:badpair}, we see that there is at most one bad pair among $\{(\phi_1,\lambda_1),\dots,(\phi_n,\lambda_n)\}$. Without loss of generality, assume that $(\phi_1,\lambda_1)$ is the only bad pair. Thus, for each $2\leq i\leq n$, there is $1\leq j_i\leq r_i-1$ such that $F_{\tau^{j_i}}^{\phi_i}(\lambda_{i})$ is expressible by the Legendre relation for $\phi_i$. It follows that
        \begin{align*}
            \dim\Gamma_M&=\trdeg_{\oK}\oK\left(\cup_{i=1}^n\left\{\frac{\lambda_{i}}{\widetilde{\pi}},\frac{F^{\phi_i}_{\tau}(\lambda_{i})}{\widetilde{\pi}},\dots,\frac{F^{\phi_i}_{\tau^{r_i-1}}(\lambda_{i})}{\widetilde{\pi}}\right\}\right)\\
            &=\trdeg_{\oK}\oK\left(\left\{\frac{\lambda_{1}}{\widetilde{\pi}},\dots,\frac{F^{\phi_1}_{\tau^{r_1-1}}(\lambda_{1})}{\widetilde{\pi}}\right\}\cup\cup_{i=2}^n\left\{\frac{\lambda_{i}}{\widetilde{\pi}},\dots,\widehat{\frac{F^{\phi_i}_{\tau^{j_i}}(\lambda_{i})}{\widetilde{\pi}}},\dots,\frac{F^{\phi_i}_{\tau^{r_i-1}}(\lambda_{i})}{\widetilde{\pi}}\right\}\right)\\
            &=(r_1+\cdots+r_n)-(n-1).
        \end{align*}
        Here, again by the notation $\widehat{F_{\tau^j}(\lambda)}/\widetilde{\pi}$, we mean that the element $F_{\tau^j}(\lambda)/\widetilde{\pi}$ is not a generator of the field described. 
        In particular, by counting the number of generators, we see that the set
        \[
            \mathfrak{T}:=\left\{\frac{\lambda_{1}}{\widetilde{\pi}},\dots,\frac{F^{\phi_1}_{\tau^{r_1-1}}(\lambda_{1})}{\widetilde{\pi}}\right\}\cup\cup_{i=2}^n\left\{\frac{\lambda_{i}}{\widetilde{\pi}},\dots,\widehat{\frac{F^{\phi_i}_{\tau^{j_i}}(\lambda_{i})}{\widetilde{\pi}}},\dots,\frac{F^{\phi_i}_{\tau^{r_i-1}}(\lambda_{i})}{\widetilde{\pi}}\right\}
        \]
        is of cardinality $(r_1+\cdots+r_n)-(n-1)$. Hence it forms a transcendence basis of $\oK(\cup_{i=1}^nP_{\phi_i})$. As $\mathfrak{S}\subset\mathfrak{T}$, we conclude that
        the set $\mathfrak{S}$ is algebraically independent over $\oK$ as desired. This completes the proof.
    \end{proof}

    \begin{proof}[Proof of Theorem~\ref{Thm:Intro}]
        By Theorem~\ref{Thm:Special_Values_to_Periods}, we have
        \[
            F(\bomega_i)\sim\left(\frac{\lambda_{\bomega_i}}{\widetilde{\pi}}\right)^\ell
        \]
        for each $1\leq i\leq n$. Then Theorem~\ref{Thm:Algebraic_Independence} implies that
        \begin{align*}
            \trdeg_{\oK}\oK(F(\bomega_1),\dots,F(\bomega_n))=\trdeg_{\oK}\oK(\lambda_{\bomega_1}/\widetilde{\pi},\dots,\lambda_{\bomega_n}/\widetilde{\pi})=n.
        \end{align*}
        
    \end{proof}

\bibliographystyle{alpha}

\end{document}